\documentclass{amsart}

\usepackage{amssymb, amsmath}
\usepackage{mathrsfs}
\usepackage{amscd}
\usepackage{verbatim}
\usepackage{enumerate}
\usepackage{a4wide}
\allowdisplaybreaks

\usepackage{color}

\theoremstyle{plain}
\newtheorem{theorem}{Theorem}[section]
\theoremstyle{remark}
\newtheorem{remark}[theorem]{Remark}

\theoremstyle{plain}

\newtheorem{lemma}[theorem]{Lemma}
\newtheorem{proposition}[theorem]{Proposition}

\newtheorem{assumption}[theorem]{Assumption}

\numberwithin{equation}{section}

\def\R{{\mathbb R}}

\newcommand{\lin}{\mathscr{L}}
\newcommand{\lf}{\mathscr{L}^\#}
\newcommand{\lfa}{\mathscr{L}^{\#,*}}
\newcommand{\lv}{{}_V\langle}
\newcommand{\lvs}{{}_{V^*}\langle}
\newcommand{\rv}{\rangle_V}
\newcommand{\rvs}{\rangle_{V^*}}

\renewcommand{\O}{\Omega}

\newcommand{\eps}{\varepsilon}

\newcommand{\lb}{\langle}
\newcommand{\rb}{\rangle}

\newcommand{\trace}{\operatorname{tr}}
\newcommand{\lt}{L^2(1+\rho)}
\newcommand{\ho}{H^1(1+\rho)}
\newcommand{\hod}{H^{1,\ast}(1+\rho)}

\begin{document}

\title[Multiscale-Analysis of stochastic RDEs]{A Multiscale-Analysis of Stochastic Bistable Reaction-Diffusion Equations}

{\author{J. Kr\"uger}
\thanks{Institut f\"ur Mathematik, Technische Universit\"at Berlin, 10623 Berlin, Germany (jkrueger@math.tu-berlin.de). The work of this author was suppported by the DFG RTG 1845.}}

{\author{W. Stannat}
\thanks{Institut f\"ur Mathematik, Technische Universit\"at Berlin, 10623 Berlin, and Bernstein Center for Computational Neuroscience Berlin, Germany (stannat@math.tu-berlin.de). The work of this author was supported by the BMBF, FKZ01GQ1001B}}
\date\today


\begin{abstract}
A multiscale analysis of 1D stochastic bistable reaction-diffusion equations 
with additive noise is carried out w.r.t. travelling waves within the variational approach to 
stochastic partial differential equations. It is shown with explicit error estimates on appropriate 
function spaces that up to lower order w.r.t. the noise amplitude, the solution can be decomposed 
into the orthogonal sum of a travelling wave moving with random speed and into Gaussian 
fluctuations. A stochastic differential equation describing the speed of the travelling wave and a 
linear stochastic partial differential equation describing the fluctuations are derived in terms of 
the coefficients. Our results extend corresponding results obtained for stochastic neural field 
equations to the present class of stochastic dynamics. 
\end{abstract}

\maketitle

\section{Introduction}
Consider the reaction-diffusion equation 
\begin{equation} 
\label{rde0} 
\partial_t v(t,x) = \nu \partial_{xx}^2 v(t,x) + bf(v(t,x))\,  , t > 0\, , x\in\Bbb R 
\end{equation} 
for strictly positive constants $\nu$, $b > 0$ and bistable reaction term   
\begin{equation}
\tag{\bf A1}
\begin{aligned}
& f (0)  = f(a) = f (1) = 0  \quad \text{ for some }  a \in (0,1) \\
& f(x) < 0 \quad \text{ for }  x \in (0,a) \, , f(x) > 0 \text{ for } x \in (a,1) \\
& f^\prime (0) < 0 , f^\prime (a) > 0, f^\prime (1) < 0\, .  
\end{aligned}
\end{equation}  
It is well-known that under the above assumptions \eqref{rde0} admits a travelling 
wave solution, i.e. a monotone increasing $C^2$-function $\hat{v}$, connecting the stable 
fixed points $0$ and $1$ of the reaction term, satisfying 
\begin{align} 
\label{tw}
c\,  \hat{v}_x = \nu\, \hat{v}_{xx} + bf (\hat{v}) 
\end{align}
for some wavespeed $c\in\R$ and boundary conditions $\hat{v} (- \infty) = 0$, $\hat{v} (+ \infty) = 
1$, see, e.g. Theorem 12 in \cite{HR}. It is easy to see that
$\hat{v} (\cdot + ct)$ and all spatial translates $\hat{v} (\cdot + x_0 + ct)$ are solutions of 
\eqref{rde0}. In the particular case of the cubic nonlinearity ${f(v) = v (1-v)(v -a)}$, equation 
\eqref{rde0} reduces to the well-known Nagumo equation (cf. \cite{Nag}), for which the travelling 
wave is explicitly given by $\hat{v}(x) = \left( 1 + e^{- \sqrt{\frac{b}{2\nu}} x}\right)^{-1}$ 
propagating along the axis with speed $c=\sqrt{2\nu b}\, (\frac12-a)$.

\medskip 
\noindent 
It can be observed in numerical simulations that the travelling wave solution persists up to 
apparently lower order fluctuations also under the impact of noisy perturbations to \eqref{rde0}. 
The term "lower-order" here is to be understood w.r.t. the order of the noise amplitude and will 
be made precise in our subsequent analysis. Even more holds for small noise amplitude: similar to 
the deterministic case, the solution of \eqref{rde0} will converge quickly to some profile of the 
type 
\begin{equation} 
\label{decomposition1}  
v(t,x) = \hat{v} (x + \gamma (t)) + u (t,x) \, , 
\end{equation}  
where $\gamma (t)$ is the random position of the wave front and $u$ denotes lower order 
fluctuations. For this reason we will call the solution $v$ also a stochastic travelling wave. 
It is the main purpose of this paper to rigorously derive a decomposition of type \eqref{decomposition1} in the case of 
noisy perturbations induced by function-space valued additive Wiener noise 
(see e.g. \cite{DaPr}), together with an identification of $\gamma$ (resp. $u$) as solution of a stochastic differential equation (resp. linear stochastic partial differential equation) accompanied with explicit error estimates in terms of $u$ in appropriate function spaces. More specifically, we consider the stochastic reaction-diffusion equation 
\begin{align}\label{RDE}
d v(t,x) = \nu \partial_{xx}^2 v(t,x) \,dt+ b f(v(t,x))\, dt + \varepsilon \,  dW(t,x),  \ \ t>0, \ \ x\in\R
\end{align}
where $W(t), \ t\in[0,T],$ denotes a $Q$-Wiener process, taking values on a 
suitable Hilbert space $H$ and $\varepsilon > 0$ will be considered to be a small parameter. 

\medskip 
\noindent 
The problem is obviously connected with the stability properties of the travelling wave solution 
$\hat{v}$ and the latter problem has been intensively studied in the deterministic setting for 
a long time, mainly based on maximum principle and comparison techniques (cf. in particular 
\cite{Fi}), and on spectral considerations (cf. \cite{Ev, Henry} and the more recent monograph 
\cite{ErmTer}). Both approaches however are not easy to carry over to the stochastic case, or 
even worse, require unnatural monotonicity conditions on the noise terms. Instead, we rather 
apply a pathwise stability analysis in the spirit of the classical Lyapunov approach to dynamical systems, first developed in \cite{Stannat12, Stannat14}.

\medskip 
\noindent 
Note that decomposition \eqref{decomposition1} is not well-posed without further assumptions, since it involves two unknowns: the position $\gamma (t)$ of the wave front and the remainder $u(t, x)$. 
In addition, there are various possibilities to define the position of the wave front. Since pointwise criteria like, e.g. level sets of $v(t,x)$ do not make much sense in the stochastic case, because of (spatially non-monotone) fluctuations, we take a minimiser 
$$ 
\gamma (t) := \mbox{ argmin}_{\gamma\in\mathbb R} \|v(t, \cdot ) - \hat{v} (\cdot + \gamma )\| 
$$ 
of the distance between $v(t, \cdot )$ and the spatial translates of $\hat{v}$ as the (in general non-unique) definition of the position of the wave-front. 

\medskip 
\noindent 
Since global minima are difficult to characterise and difficult to handle in the context of 
dynamical equations, we will require only the necessary condition $\langle v (t, \cdot ) - \hat{v} 
(\cdot + \gamma (t) ) , \hat{v}_x (\cdot + \gamma (t))\rangle = 0$ for local minima in the decomposition 
\eqref{decomposition1}. Here, $\langle  \cdot , \cdot \rangle$ denotes the scalar product of the 
underlying Hilbert space $H$. In \cite{InLa}, the authors used this approach and obtained results 
on the corresponding decomposition up to the first random time $\tau$, when the local minimum 
becomes a saddle point. Apart from the necessity of introducing the above mentioned stopping time, 
the position of the wave front will be a semi-martingale only, in particular non-differentiable. In 
addition, the decomposition obtained in this way will be non-explicit w.r.t. the small parameter 
$\varepsilon$.   

\medskip 
\noindent 
To avoid non-differentiability we will consider in this paper the following (differentiable) 
approximation $\gamma^m (t)$, defined as the solution of the following pathwise ordinary 
differential equation 
$$ 
\dot{\gamma}^m (t) = c + \langle v( t, \cdot) - \hat{v} (\cdot + \gamma^m (t)), \hat{v}_x (\cdot + 
\gamma^m (t))\rangle  
$$
for suitable initial condition and a priori chosen relaxation parameter $m > 0$. We will then show in Theorem \ref{thm:firstorder} as our first main result that $\gamma^m (t)$ admits the decomposition 
$$ 
\gamma^m (t) = ct + \varepsilon \, C^m_0 (t) +  o(\varepsilon )
$$ 
where $C^m_0 (t) = \int_0^t c_0^m (s)\, ds $, and $c^m_0$ is the unique strong solution of the stochastic ordinary differential equation
$$ 
dc^m_0 (t) =  - m c_0^m (t) \, dt + m\langle \hat{v}_x , dW(t)\rangle 
$$ 
for suitable initial condition. The fluctuations $u(t) = v(t, \cdot ) - \hat{v} (\cdot + ct + \varepsilon C_0^m (t)))$ can be represented as 
$$ 
u(t) = \varepsilon u_0^m (t) + o(\varepsilon) 
$$ 
where $u_0^m$ is the unique variational solution of the (affine) linear stochastic partial differential equation 
$$ 
du_0^m (t) = \left[ \nu\Delta u_0^m (t) + bf^\prime (\hat{v}( \cdot + ct)) u_0^m (t) - c_0^m (t) \hat{v}_x (\cdot + ct))\right]\, dt + dW(t) \, . 
$$ 

\medskip 
\noindent 
To compare our results to the analysis in \cite{InLa}, based on the variational characterisation of local minima of $\|v(t, \cdot ) - \hat{v} (\cdot + \gamma )\|$, we also consider the asymptotics  
of our decomposition for $m\to\infty$, i.e. in the limit of immediate relaxation. We will show in Lemma \ref{lem:conv} below that both processes $\gamma^m (t)$ and $u_0^m (t)$ converge as $m\to\infty$ and we will identify the limiting processes again as solutions of stochastic (partial) differential equation together with explicit error estimates in the associated decomposition \eqref{decomposition1} (see Theorem \ref{thm:limit}).

\section{Stochastic Reaction-Diffusion Equations in Weighted Sobolev Spaces}\label{sec:RDE}

\subsection{Realisation as stochastic evolution equation}

Thinking of a typical trajectory of a stochastically perturbed travelling wave one cannot expect the solution $v$ of \eqref{RDE} to take values in $L^2(\R)$, which raises the question in which sense and especially on which function space to model the above reaction-diffusion equation. If we assume the noise to be on $L^2(\R)$, what is true however, is that the difference between the solution $v$ and a deterministic reference profile $\hat{v}$ takes values in $L^2(\R)$. This motivates us to derive a decomposition of $v$ into $v=u+\hat{v}$ and investigate properties of the difference $u$ under a smallness condition on $u^0 :=v^0 - \hat{v}$. This approach to study the dynamics of stochastic travelling wave solutions has also been used in \cite{BrWe},\cite{Stannat12}, \cite{KrSt} and \cite{Lang}. In the stability analysis below it turns out, that controlling $u$ in $L^2(\R)$ with respect to the Lebesgue measure is not sufficient, but that we need an additional control with respect to the measure 
\[
\rho(x)=Z \, e^{-\frac c\nu x}\, 
\]
with a positive constant $Z>0$. This measure is naturally associated to the equation and will be derived and motivated in Section \ref{sec:frozen}. Assume $(W(t))_{t\in[0,T]}$ to be a $Q$-Wiener process on the space $L^2(1+\rho)=L^2(\R, dx)\cap L^2(\R, \rho\, dx)$. First, we rewrite \eqref{RDE} in terms of $u=v-\hat{v}$ and obtain
\begin{align}\label{RDEu}
 du(t,x) = \nu \Delta u(t,x)\,dt + b \big(f(u(t,x)+ \hat{v}(x+ct)) - f(\hat{v}(x+ct))\big)\, dt + \varepsilon dW(t,x),  \ \ t>0, \ \ x\in\R
\end{align}
This SPDE can now be formulated as a stochastic evolution equation on $L^2(1+\rho)$. Following the approach of \cite{PreRo} we choose the Gelfand triple 
\[
H^1(1+\rho)\hookrightarrow L^2(1+\rho)\hookrightarrow H^{1,\ast}(1+\rho)
\] 
with the norms
\[
\Vert u\Vert_{1+\rho}^2= \int u^2 \, (1+\rho)\, dx \quad \text{ on } L^2(1+\rho)
\]
and
\[
 \|u\|_{H^1(1+\rho)}^2 = \|u\|_{1+\rho}^2 + \| u_x\|_{1+\rho}^2   \quad \text{ on } H^1(1+\rho).
 \] 
Define the operator
\mbox{$\Delta:H^1(1+\rho)\to H^{1,\ast}(1+\rho)$} as
\[{}_{H^{1,\ast}}\lb \Delta v, u \rb_{H^{1}}= -\int_{\R}  v_x \,  \big(u \, (1+\rho)\big)_x \, dx \]
Let $G:\R\times H^{1}(1+\rho)\to H^{1,\ast}(1+\rho)$ be given by
\[{}_{H^{1}}\lb v, G(t,u)\rb_{H^{1,\ast}}  = \int_{\R} v(x)\, b \big(f(u(x) + \hat{v}(x+ct)) - f(\hat{v}(x+ct))\big) (1+\rho(x))\, dx.\]

In this terminology \eqref{RDEu} corresponds to the stochastic evolution equation 
\begin{equation}\label{eq:SEu}
\begin{aligned}
d u &= \nu \Delta u \, dt + G(t,u)\, dt + \varepsilon \, d W
\\ u(0) & = u^0.
\end{aligned}
\end{equation}
on $L^2(1+\rho)$. To prove well-posedness of this equation we impose the following additional assumptions on the global behaviour of the reaction term $f$.
We assume that the derivative $f^\prime$
is bounded from above by
\begin{equation} 
\tag{\bf B1} 
\eta_1 := \sup_{x\in\R} f^\prime (x) < \infty   \, , 
\end{equation} 
that there exists a finite positive constant $L$ such that 
\begin{equation} 
\tag{\bf B2} 
\left| f(x_1 ) - f(x_2)\right| \le L |x_1 - x_2| \left( 1 + x_1^2 + x_2^2 \right) \qquad \forall x_1 , x_2 \in \R\, , 
\end{equation} 
which is typically satisfied for polynomials of third degree with leading negative coefficient and that there exists $\eta_2$ such that  
\begin{equation}
\tag{\bf B3} 
\left| f(u+v) - f(v) - f^\prime (v) u \right|  \le \eta_2 (1 + |u|) |u|^2 \qquad \forall v\in [0,1]\, , u\in \R\, . 
\end{equation} 
Furthermore, to ensure even higher regularity of solutions of \eqref{eq:SEu}, namely weak spatial differentiability, we demand the following growth condition of the derivative $f'$: There exists a constant $\eta_3 > 0$ with
\begin{align*}
\tag{\bf B4} 
&\vert f'(u+v)\vert \leq \eta_3 (1+\vert u\vert^2)\\[3pt]
\text{and }&\vert f'(u+v)-f'(v)\vert \leq \eta_3 (\vert u\vert +\vert u\vert^2)
\qquad \forall v\in [0,1]\, , u\in \R\,.
\end{align*}

\noindent 
For simplicity of the following analysis we also assume
\begin{equation}
\tag{\bf A1}
\int_0^1 f(v)\, dv \ge  0 
\end{equation} 
This condition on the input parameter $f$ or, in other words, on the associated potential 
$F=\int f\, dv$ implies that the wave speed $c$, as part of the travelling wave solution 
$\hat{v}$, is nonnegative. To briefly illustrate this relation we multiply \eqref{tw} by 
$\hat{v}_x$ and integrate over the real line to obtain 

\begin{align*}
c\int \hat{v}_x^2 \, dx& = \nu \int \hat{v}_{xx}\, \hat{v}_x \, dx + b \int f(\hat{v})\, \hat{v}_x \, dx
= \frac \nu 2\int \frac{d}{dx} \hat{v}_x^2 \, dx + b \int_0^1 f(\hat{v})\, d\hat{v} 
=  b \int_0^1 f(\hat{v})\, d\hat{v}
\end{align*}
due to the boundary conditions of $\hat{v}$ at $\pm \infty$. This shows that the wave speed $c$ and 
the integral on the right-hand side have the same sign. Assumption (\textbf{A1}) is required in our 
analysis only to ensure that $\varrho$ is monotone decreasing. Dropping this assumption would 
require to consider both cases, monotone decreasing (resp. increasing).

\medskip
\noindent 
Under the above conditions the following existence and uniqueness result can be stated
\begin{theorem}\label{thm:solu} Assume $(\textnormal{\bf B1})-(\textnormal{\bf B3})$.
For each $T>0, \, \eps >0$ and each $\eta\in L^2(1+\rho)$ there exists a unique variational solution $u\in L^2(\O;C([0,T];\lt))\cap L^2(\O\times [0,T];\ho)$, almost surely satisfying the integral equation
\begin{align}\label{eq:solu}
u(t)&= u^0+ \int_0^t \nu \Delta u(s) \, ds + \int_0^t G(s,u(s))\, ds + \varepsilon W(t),  \ \ \ t\in [0,T] \\\notag
u^0&=\eps \eta
\end{align}
Here the Bochner integrals are defined in $\hod$. Moreover, there exists a constant\newline
\mbox{$C(T,\omega, \Vert u^0 \Vert^2_{1+\rho})>0$} such that for all $t\in[0,T]$ 
\begin{align}\label{eq:ito}
\sup_{s\leq t} \Vert u(s)\Vert_{1+\rho}^2 + \int_0^t \Vert u(s)\Vert_{H^1(1+\rho)}^2\, ds \leq C(T,\omega, \Vert u^0 \Vert^2_{1+\rho})
\end{align}
and $C(T,\omega, \Vert u^0 \Vert^2_{1+\rho})<\infty$ for almost all $\omega\in\Omega$.
\end{theorem}
\begin{proof}
We show that the coefficients in \eqref{eq:solu} satisfy the conditions of Theorem 1.1 in \cite{LiuRo}. Let $u,w\in \ho$. $G$ can be realised as a continuous mapping $\ho \to \hod$: Using property $({\bf B2})$ and the elementary estimate $\Vert u\Vert_\infty \leq \Vert u\Vert_{\ho}$ one obtains
\begin{align*}
{}_{H^{1,\ast}}\lb G(t,u), w \rb_{H^{1}} &= b\, \int \big(f(u+\hat{v}(\cdot+ct))-f(\hat{v}(\cdot+ct))\big) \, w\,(1+\rho) \, dx\\
&\leq bL \int \vert u\vert (3+2u^2) \vert w\vert (1+\rho)\, dx
\leq bL (3+2\Vert u\Vert_{\ho}^2) \Vert u\Vert_{\ho} \Vert w\Vert_{\ho} 
\end{align*}
The drift and diffusion operators are clearly hemicontinuous. For showing coercivity we estimate
\begin{align*}
{}_{H^{1,\ast}}\lb \nu\Delta u, u\rb_{H^{1}} = -\nu \int u_x^2 (1+\rho) \, dx - \nu \int u_x u \, \rho_x \, dx
\leq -\nu \Vert u\Vert_{\ho}^2+\left(\nu+\frac {c^2}{2\nu}\right)\Vert u\Vert_{1+\rho}^2
\end{align*}
and
\begin{align*}
{}_{H^{1,\ast}}\lb G(t,u),u\rb_{H^{1}} = b\int \big(f(u+\hat{v}(\cdot+ct))-f(\hat{v}(\cdot+ct))\big)\, u \,(1+\rho)\, dx \leq b\, \eta_1 \, \Vert u\Vert^2_{1+\rho}
\end{align*}
since $f$ is one-sided Lipschitz continuous, i.e. $(f(x)-f(y))(x-y)\leq \eta_1 (x-y)^2$ for all $x,y\in \R$ using $({\bf B1})$. Weak monotonicity of the linear part of the drift operator is covered by the coercivity condition. For the nonlinear part with the  one-sided Lipschitz estimate from above one obtains for $u_1, u_2\in H^1(1+\rho)$
\begin{align*}
{}_{H^{1,\ast}}\lb G(t,u_1)-G(t,u_2), u_1-u_2\rb_{H^{1}} &  = b \int \big(f(u_1+\hat{v}(\cdot+ct))-f(u_2+\hat{v}(\cdot+ct))\big) (u_1-u_2) \, (1+\rho) \, dx \\
&\leq b\, \eta_1 \int (u_1-u_2)^2 (1+\rho) \, dx = b\, \eta_1 \Vert u_1-u_2\Vert^2_{1+\rho}
\end{align*}
Furthermore, $G$ is of admissible growth: Let $w\in \ho$ with \mbox{$\Vert w\Vert_{\ho}\leq 1$}. Applying condition $({\bf B2})$ yields
\begin{align*}
{}_{H^{1,\ast}}\lb G(t,u),w\rb_{H^{1}} &\leq bL\int \vert u\vert (3+2u^2)\vert w\vert (1+\rho)\, dx\leq 3bL \Vert u\Vert_{1+\rho} +2bL \Vert u\Vert_{\ho}\Vert u\Vert_{1+\rho}^2\\
&\leq 3bL \Vert u\Vert_{\ho} (1+\Vert u\Vert_{1+\rho}^2)
\end{align*}

For the second part of the statement an application of It\^o 's formula as stated in \cite[Theorem 4.2.5]{PreRo} together with the coercivity of the drift yields for all $t\in[0,T]$
\begin{align}\label{eq:est}\notag
\Vert u(t)\Vert_{1+\rho}^2 &= \Vert u^0\Vert_{1+\rho}^2 + \int_0^t \Big( 2\, {}_{H^{1,\ast}}\lb \nu\Delta u(s)+G(s,u(s)), u(s)\rb_{H^1} \, ds + \Vert \sqrt{Q}\Vert_{L_2(L^2(1+\rho))}^2\Big)\, ds\\\notag
&\quad +2 \int_0^t \lb u(s), dW(s)\rb\\
&\leq \Vert u^0\Vert_{1+\rho}^2 -2\nu \int_0^t \Vert u(s)\Vert_{H^1(1+\rho)}^2\, ds + 2\left(\nu+\frac{c^2}{2\nu} + b\eta_1\right)\int_0^t \Vert u(s)\Vert_{1+\rho}^2\, ds \\\notag
&\quad+ T \trace(Q) + M(t)
\end{align} 
where $M(t),t\in[0,T],$ denotes the martingale part and $\Vert \cdot \Vert_{L_2(L^2(1+\rho))}$ the Hilbert-Schmidt norm on $L^2(1+\rho)$. In particular, setting $\theta:= 2\left(\nu+\frac{c^2}{2\nu} + b\eta_1\right)$ this implies
\begin{align*}
\Vert u(t)\Vert_{1+\rho}^2 &\leq \Vert u^0\Vert_{1+\rho}^2 + \theta \int_0^t \Vert u(s)\Vert_{1+\rho}^2\, ds + T \trace(Q) + M(t)
\end{align*}  
such that Gronwall's lemma allows us to bound the expectation of the left-hand side by
\begin{align}\label{gron}
E\left[\Vert u(t)\Vert_{1+\rho}^2\right]\leq e^{\theta t\,}(\Vert u^0\Vert_{1+\rho}^2 + T \trace(Q))\, .
\end{align}
Taking first the supremum and then the expectation of all terms in \eqref{eq:est} and inserting \eqref{gron} we obtain the estimate
\begin{align*}
&E\left[\sup_{t\in[0,T] }\Vert u(t)\Vert_{1+\rho}^2 +2\nu  \int_0^T
 \Vert u(s)\Vert_{H^1(1+\rho)}^2\, ds\right] \\
 &\leq \Vert u^0\Vert_{1+\rho}^2  + \theta \int_0^T E\left[\Vert u(s)\Vert_{1+\rho}^2\right]\, ds + T \trace(Q) 
 + E\left[\sup_{t\in[0,T]} \vert M(t)\vert\right]\\
 &\leq \Vert u^0\Vert_{1+\rho}^2  + e^{\theta T} \, (\Vert u^0\Vert_{1+\rho}^2 + T \trace(Q)) + T \trace(Q) 
 + E\left[\sup_{t\in[0,T]} \vert M(t)\vert\right]
\end{align*}
A control of the last expectation in terms of the martingale's quadratic variation $[M]_t, t\in[0,T], $ is provided by the Burkholder-Davis-Gundy inequality
\begin{align*}
E\left[\sup_{t\in[0,T]} \vert M(t)\vert\right] &\leq C E\left[ [M]_t^{\frac 12}\right]= 2C E\left[\left(\int_0^t \lb u(s), Q u(s)\rb \, ds\right)^{\frac 12}\right]\\
& \leq 2C{\Vert Q\Vert^{\frac 12}_{\mathscr{L}(L^2(1+\rho))}}\,  E\left[\left(\int_0^t \Vert u(s)\Vert_{1+\rho}^2 \, ds\right)^{\frac 12}\right]\\
& \leq 2 C{T^{\frac 12}\, \Vert Q\Vert^{\frac 12}_{\mathscr{L}(L^2(1+\rho))}}\,  E\left[\sup_{t\in[0,T]} \Vert u(t)\Vert_{1+\rho}\right]\\
&\leq  2 C {T^{\frac 12}\, \Vert Q\Vert^{\frac 12}_{\mathscr{L}(L^2(1+\rho))}}\,  E\left[\sup_{t\in[0,T]} \Vert u(t)\Vert_{1+\rho}^2\right]^{\frac 12}\\
&\leq 2 C {T^{\frac 12}\, \Vert Q\Vert^{\frac 12}_{\mathscr{L}(L^2(1+\rho))}} \, e^{\frac \theta 2 T\,}(\Vert u^0\Vert_{1+\rho}^2 + T \trace(Q))^{\frac 12}
\end{align*}
Summarizing the above estimates, there exists a positive constant $C(T,\Vert u^0\Vert^2_{1+\rho})$ such that 
\begin{align*}
E\left[\sup_{t\in[0,T] }\Vert u(t)\Vert_{1+\rho}^2 +2\nu  \int_0^T
 \Vert u(s)\Vert_{H^1(1+\rho)}^2\, ds\right] \leq C(T,\Vert u^0\Vert^2_{1+\rho})
\end{align*}
In particular, there exists a path-dependent constant $C(T,\omega, \Vert u^0\Vert^2_{1+\rho})$ with
\begin{align*}
\sup_{t\in[0,T] }\Vert u(t)\Vert_{1+\rho}^2 +2\nu  \int_0^T
 \Vert u(s)\Vert_{H^1(1+\rho)}^2\, ds \leq C(T, \omega, \Vert u^0\Vert^2_{1+\rho})
\end{align*}
and $C(T,\omega, \Vert u^0\Vert^2_{1+\rho})<\infty$ for $P$-almost every $\omega\in\Omega$.
\end{proof}

\noindent 
From the above decomposition it follows that $v := u+ \hat{v}$ is a solution of the original reaction-diffusion equation \eqref{RDE}.

\subsection{Continuity of the solution in $H^1(1+\rho)$}
We will now show that, assuming higher spatial regularity of the noise as well as an additional growth condition of $f'$, the unique solution $u$ of \eqref{eq:solu} is even continuous with values in $H^1(1+\rho)$.

\begin{assumption}\label{ass:noise}
Let $\sqrt{Q}$ be a Hilbert-Schmidt operator from $L^2(1+\rho)$ to $H^1(1+\rho)$.
\end{assumption}

\begin{remark}\label{lem:rho} In the following analysis we will frequently consider translations 
of the measure $\rho$ w.r.t. deterministic as well as stochastic shifts, for which reason the following two properties of $\rho$ will become important: 
\begin{enumerate}[(i)]
\item $\rho$ is decreasing, i.e. $\rho(x+y)\leq\rho(x)$ for all $y>0$. 
\item $\rho$ is of (at most) exponential growth, i.e., $\rho(x-\xi)\leq e^{M\vert \xi\vert}\rho(x)$ 
for $M= \frac c\nu$ and all $\xi\in\R$\, . 
\end{enumerate}
\end{remark}

\begin{lemma}\label{lem:grad} Assume \textnormal{({\bf B4})}. Let $p_t = e^{t\nu\Delta}$, $t\ge 0$, denote the semigroup generated by $\nu\Delta$. Then there exists a finite constant $C=C(T, \Vert \hat{v}_x\Vert_\infty)$ such that for all $u\in H^1(1+\rho)$ and for all $s\leq t$ the following bound holds: 
\begin{align*}
\Vert \partial_x p_{t-s}G(s,u)\Vert_{1+\rho}^2 \leq C (1+\Vert u\Vert_{1+\rho}^2 \Vert u_x\Vert_{1+\rho}^2) \, (\Vert u\Vert_{1+\rho}^2+\Vert u_x\Vert_{1+\rho}^2)\, . 
\end{align*}
\end{lemma}

\begin{proof}
Let $u\in C_c^1(\R)$. Then, \textnormal{({\bf B4})} leads to  
\begin{align*}
&\vert \partial_x p_{t-s} G(s,u)\vert = \vert p_{t-s}\left(\partial_x G(s,u)\right)\vert \\
&\leq p_{t-s}\big\vert f'(u+\hat{v}(\cdot +cs))\, u_x + (f'(u+\hat{v}(\cdot +cs))-f'(\hat{v}(\cdot +cs)))\, \hat{v}_x(\cdot+cs)\big\vert\\
&\leq p_{t-s}\big[\eta_3 \, (1+u^2)\vert u_x\vert + \eta_3 \, (\vert u\vert + u^2)\,  \hat{v}_x(\cdot+cs)\big]
\end{align*}
It is well-known that $\nu\Delta$ generates the Gaussian semigroup
\begin{align*}
p_tf(x)=\frac 1{\sqrt{4\pi\nu t}} \int f(x+y) e^{-\frac{y^2}{4\nu t} }\, dy\,. 
\end{align*}
Applying it to the measure $1+\rho$ we obtain
\begin{align}\label{ptrho}
p_t (1+\rho) (x) = 1+\frac{Z e^{-\frac c\nu x}}{\sqrt{4\pi\nu t}} \int e^{-\frac {y^2}{4\nu t}-\frac {c}{\nu} y} \, dy 
= 1+\frac{Z e^{-\frac c\nu x + \frac{c^2}{\nu} t }}{\sqrt{4\pi \nu t}} \int e^{-\frac{(y+ 2c t)^2}{4\nu t}} \, dy 
= 1+Z e^{-\frac c\nu x + \frac{c^2}{\nu} t }\, . 
\end{align}

Hence, there exists a constant $C=C(\Vert \hat{v}_x\Vert_\infty)$ such that 
\begin{align*}
\int \vert \partial_x p_{t-s} G(s,u)\vert^2 \, (1+\rho)\, dx
&\leq C \int p_{t-s}\big(u_x^2+u^4u_x^2 +u^2+u^4\big)\, (1+\rho)\, dx\\
&= C \int \big(u_x^2+u^4u_x^2 +u^2+u^4\big)\, (1+p_{t-s}\rho)\, dx\\
&\leq C\, e^{\frac{c^2}{\nu}(t-s)}\int \big(u_x^2+u^4u_x^2 +u^2+u^4\big)\, (1+\rho)\, dx\\
&\leq C \, e^{\frac{c^2}{\nu}T}(1+\Vert u\Vert_{L^\infty(dx)}^4)\int (u_x^2+u^2)\,(1+\rho)\, dx
\end{align*}
Now, with the elementary estimate $\Vert u\Vert_{L^\infty(dx)}^2\leq 2 \, \Vert u \Vert_{1+\rho}\Vert u_x\Vert_{1+\rho}$ there exists a positive constant  $C=C(T, \Vert \hat{v}_x\Vert_\infty)$ such that
\begin{align*}
\Vert \partial_x p_{t-s}G(s,u)\Vert_{1+\rho}^2 &\leq C (1+\Vert u\Vert_{1+\rho}^2 \Vert u_x\Vert_{1+\rho}^2) \, (\Vert u\Vert_{1+\rho}^2+\Vert u_x\Vert_{1+\rho}^2)\, . 
\end{align*}
\end{proof}

\noindent With this at hand we are able to formulate the following regularity statement:
\begin{proposition}\label{prop:ex}
Under Assumption \ref{ass:noise} for each $T>0$ and each $\eta\in H^1(1+\rho)$ there exists a unique variational solution $u\in C([0,T];H^1(1+\rho))$ of equation \eqref{eq:solu}. 
\end{proposition}
\begin{proof}
Let $u$ be the unique variational solution in \mbox{$L^2(\O;C([0,T];\lt))\cap L^2(\O\times (0,T);\ho)$} from Theorem \ref{thm:solu}. Using the semigroup $p_t = e^{t\nu\Delta}$, $t\ge 0$, $u$ allows for the mild solution representation 
\begin{align*} 
u(t)=p_t u^0 + \int_0^t p_{t-s}G(s, u(s))\, ds + \int_0^t p_{t-s}dW(s)
\end{align*}
We will show that indeed $u$ is differentiable in space and the gradient takes values in $\lt$: For every $t\in[0,T]$ with Lemma \ref{lem:grad}
\begin{align*}
&\Vert u_x(t)\Vert_{1+\rho}^2 \leq 2\Vert p_t u^0_x\Vert_{1+\rho}^2 + 4 \int_0^t \Vert \partial_x p_{t-s}G(s,u(s))\Vert_{1+\rho}^2 \, ds + 4 \, \Big\Vert \partial_x \int_0^t p_{t-s}dW(s)\Big\Vert_{1+\rho}^2\\
&\leq 2\Vert p_t u^0_x\Vert_{1+\rho}^2 + 4 C \int_0^t  (1+\Vert u(s)\Vert_{1+\rho}^2 \Vert u_x(s)\Vert_{1+\rho}^2) \, (\Vert u(s)\Vert_{1+\rho}^2+\Vert u_x(s)\Vert_{1+\rho}^2) \, ds\\
&\quad + 4 \Big\Vert \partial_x \int_0^t p_{t-s}dW(s)\Big\Vert_{1+\rho}^2 =: I+II+III, \text{ say.}
\end{align*}
Clearly, 
\[
I\leq \sup_{t\in[0,T]} \Vert p_t\Vert^2 \Vert u^0_x\Vert_{1+\rho}^2=:C_1(T,\Vert u^0_x\Vert_{1+\rho})
\]
For uniformly bounding $III$ note that for a $Q$-Wiener process $W$ on $\ho$ we know that
\[
\int_0^t p_{t-s}dW(s) \in L^2(\Omega, C([0,T], \ho))
\]
and thus 
\[
III\leq 4 \sup_{t\in[0,T]} \Big\Vert \partial_x \int_0^t p_{t-s}dW(s)\Big\Vert_{1+\rho}^2=: C_3(T,\omega) <\infty \quad \text{for P-a.e. }\omega\in\Omega\, . 
\]
The bound \eqref{eq:ito} is now applied to reduce the second term to
\begin{align*}
II\leq C_{2,1}\int_0^t \Vert u_x (s)\Vert_{1+\rho}^4 \, ds +C_{2,2}
\end{align*}
for constants $C_{2,1}=C_{2,1}(T,\omega, \Vert u^0 \Vert_{1+\rho}))$ and $ C_{2,2}=C_{2,2}(T,\omega, \Vert u^0 \Vert_{1+\rho})$. 
Altogether, this leads to
\[
\Vert u_x(t)\Vert_{1+\rho}^2\leq C(T,\omega, \Vert u^0\Vert_{H^1(1+\rho)})+C_{2,1}\int_0^t \Vert u_x (s)\Vert_{1+\rho}^4 \, ds \, .
\]
Set $\alpha(t):= \Vert u_x(t)\Vert_{1+\rho}^2$. Then, Gronwall's inequality together with \eqref{eq:ito} implies
\[
\Vert u_x(t)\Vert_{1+\rho}^2 \leq C(T,\omega, \Vert u^0\Vert_{H^1(1+\rho)})\, \exp\left({C_{2,1}\int_0^t \alpha(s)\, ds}\right)\leq \tilde{C}(T,\omega, \Vert u^0\Vert_{H^1(1+\rho)})
\]
for a positive constant $\tilde{C}$. The previous $P$-nullsets on which $\tilde{C}$ is possibly infinite do not depend on $t$ such that we even have
\[
\sup_{t\in[0,T]}\Vert u_x(t)\Vert_{1+\rho}^2 \leq  \tilde{C}(T,\omega,\Vert u^0\Vert_{H^1(1+\rho)})
\]
To prove continuity it remains to show that $u_x\in C([0,T];L^2(1+\rho))$. Considering
\begin{align*}
u_x(t)&= p_t u_x^0 + \int_0^t \partial_x p_{t-s} G(s,u(s)) \, ds + \partial_x \int_0^t p_{t-s} dW(s)= I+II+III
\end{align*}
clearly $I$ and $III$ have the required regularity. For the second term note that with Lemma \ref{lem:grad} for all $s\leq t$ the integrand takes values in $L^2(1+\rho)$ such that $II$ is a continuous process on this space.
\end{proof}

The objective of the following sections is to more thoroughly study the behaviour of $u$, providing information about the deviation of the stochastic solution from the deterministic wave. We will identify a suitable stochastic phase of the chosen reference profile $\hat{v}$ and subsequently quantify the remaining fluctuations around this profile in terms of the noise strength $\eps$. This first-order stability analysis with respect to $\eps$ follows closely the methods and ideas developed in \cite{Lang} (and partially also in \cite{KrSt}) for the study of stochastic neural field equations. In the case of bistable reaction-diffusion equations the involved diffusion and reaction operators however are no longer bounded, such that an extended notion of solution as well as refined estimates have to be introduced. 
To have the flexibility of also working in the larger spaces $L^2(\rho)$ and $L^2(\R)$ we introduce a second Gelfand triple 
$$V\hookrightarrow H=H^\ast \hookrightarrow V^\ast$$
with $H=L^2(\R)$ and $V=H^1(\R)$.

\subsection{The frozen-wave setting}\label{sec:frozen} 
Since our aim is to study the dynamics of solutions whose initial profile is already close to a travelling wave solution, we linearise the nonlinearity $f$ around the travelling wave $\hat{v}$ and obtain the representation

\begin{equation}\label{eq:SEl}
\begin{aligned}
d u(t)
&= \left[\mathscr{L}_t u(t)+ R(t,u)\right]\, dt + \varepsilon \, d W(t)
\end{aligned}
\end{equation}
with 
\begin{align}\label{def:lin}
\mathscr{L}_t: V\to V^*, \quad
\mathscr{L}_t u= \nu \Delta u+ b  f'(\hat{v}(\cdot+ct))\, u, \;\;t\in[0,T]
\end{align} 

and the nonlinear remainder
\begin{align*}
R: [0,T]\times V \to V^*, \quad
R(t,u)= b \left(f(u+\hat{v}(\cdot+ct))-f(\hat{v}(\cdot+ct))-f'(\hat{v}(\cdot+ct))u\right)
\end{align*}

To identify the dynamical equations for the wave-speed for the stochastic travelling wave, as well 
as for the remaining fluctuations, it will be useful to introduce the frozen-wave operator $\mathscr{L}^\#:V\to V^*$ given by
\begin{align*}
\mathscr{L}^\#u := \nu \Delta u + bf'(\hat{v}) u -c \partial_x u \, . 
\end{align*}
Note that by \eqref{tw} we have $\mathscr{L}^\#\hat{v}_x=0$.
Likewise, one can easily check that an eigenvector $\Psi$ of the adjoint operator corresponding to the eigenvalue $0$ is given by $\Psi(x)=e^{-\frac c \nu x}\,\hat{v}_x(x)$. $\Psi$ is strictly positive and we assume the following regularity property:

\begin{assumption}\label{ass:eigenfct}
The zero-eigenfunction $\Psi$ of $\lfa$ satisfies $\Psi \in V$.
\end{assumption}

Indeed, this regularity can be shown under the following additional assumption on the reaction function $f$, namely
\begin{equation} 
\tag{\bf A2} 
\begin{aligned} 
& f\in C^2 \, ,  f^{\prime\prime} (0) > 0 \text{ and } f^{\prime\prime} (1) < 0\, 
\end{aligned} 
\end{equation} 
saying that $f$ is strictly convex in a small neighbourhood around $0$
and strictly concave 
in a small neighbourhood around $1$.
The main implication of this additional assumption is then formulated in the next Lemma. 

\begin{lemma} 
\label{lem1_0}  
Assume $({\bf A1})-({\bf A2})$. Then $\frac{f(\hat{v})}{\hat{v}_x} (x)$ is strictly increasing 
in $x$ at $\pm\infty$. In particular, 
$$ 
\exists \gamma_- := \lim_{x\downarrow -\infty} \frac b\nu \frac{f(\hat{v})}{\hat{v}_x} (x) = \frac{c}{2\nu} - \sqrt{\left( \frac c{2\nu}\right)^2 - \frac b\nu f^\prime (0)} < 0 
$$ 
and 
$$  
\exists \gamma_+ := \lim_{x\downarrow \infty} \frac b\nu \frac{f(\hat{v})}{\hat{v}_x} (x) 
= \frac{c}{2\nu} + \sqrt{\left( \frac c{2\nu}\right)^2 - \frac b\nu f^\prime (1)} > 0 \, . 
$$ 
\end{lemma} 

\begin{proof} 
We will first show that $\lim_{|x|\to\infty} e^{-\frac c\nu x} \hat{v}_x^2 (x) = 0$. 

\smallskip
\noindent 
Indeed, $\hat{v}_x \in L^1 (\R)$, hence $\lim_{n\to\infty} \hat{v}_x (x_n) = 0$ for some sequence $x_n\uparrow\infty$, implies that 
$$ 
\begin{aligned} 
\hat{v}_x^2 (x) 
& = \hat{v}_x^2 (x_n) - 2 \int^{x_n}_{x} \hat{v}_{xx} \hat{v}_x \, dx \\
& = \hat{v}^2_x (x_n) - 2 \frac{c}{\nu} \int^{x_n}_{x}\hat{v}^2_x \, dx + 2 \frac{b}{\nu} \int^{x_n}_x f (\hat{v}) \hat{v}_x \, dx \\ 
& \le \hat{v}_x^2 (x_n) + 2 \frac{b}{\nu} \int^{\hat{v} (x_n)}_{\hat{v} (x)} f(v) \, dv  \qquad \forall n\, .  
\end{aligned} 
$$ 
Consequently, 
$$ 
\hat{v} _x^2 (x) 
\le \lim_{n\to\infty} \hat{v}^2_x (x_n) + 2 \frac{b}{\nu} \int^{\hat{v} (x_n)}_{\hat{v} (x)} f (v)\, dv 
= \frac{2b}{\nu} \int^1_{\hat{v} (x)} f(v)\, dv \, . 
$$
In particular, 
$$ 
\lim_{x\to\infty} \hat{v}_x^2 (x) \le \limsup_{x\to\infty} \frac{2b}{\nu} \int^1_{\hat{v} (x)} f (v)\, dv = 0 
$$ 
and thus, $\lim_{x\to\infty} e^{-\frac c\nu x} \hat{v}_x^2 (x) = 0$, too.

\smallskip 
\noindent 
To see that $\lim_{x\to -\infty} e^{-\frac c\nu x} \hat{v}_x^2 (x) = 0$ note that for 
$x\le x_a := \hat{v}^{-1} (a)$ 
$$
\frac{d}{dx} ( e ^{- 2\frac{c}{\nu} x } \hat{v}^2_x )(x)  
= 2\left( - \frac{c}{\nu}\hat{v}_x  + \hat{v}_{xx}\right) e^{- 2 \frac{c}{\nu} x} \hat{v}_x (x) 
=  - 2 \, \frac{b}{\nu} e^{- 2 \frac{c}{\nu} x } f (\hat{v}) \hat{v}_x (x) \ge 0\, . 
$$ 
Consequently,  
$$ 
\lim_{x\to-\infty} e^{- 2\frac{c}{\nu} x } \hat{v}^2_x (x) = \inf_{x \le \hat{v}^{-1} (a)} e^{- 2\frac{c}{\nu} x } \hat{v}^2_x (x) =: \gamma < \infty 
$$ 
and thus   
$$ 
\lim_{x\to -\infty} e^{- \frac{c}{\nu} x} \hat{v}^2_x (x) 
\le \limsup_{x\to -\infty} e^{\frac{c}{\nu} x } \gamma = 0\, . 
$$

\smallskip 
\noindent 
Now define $w := e^{-\frac c{2\nu} x} \hat{v}_x$ and note that 
\begin{equation} 
\label{eq:lem1_0} 
w_{xx} =  \left( \left( \frac c{2\nu} \right)^2 - \frac b\nu f^\prime (\hat{v})\right)w\, , 
\end{equation} 
since differentiating $c\hat{v}_x = \hat{v}_{xx} + bf(\hat{v})$ implies 
$c\hat{v}_{xx} = \hat{v}_{xxx} + bf^{\prime}(\hat{v})\hat{v}_x$. Then Assumption ({\bf A2}) implies that 
$$ 
\frac d{dx} \left( w_x^2  + \left( \frac b\nu f' (\hat{v}) - \left( \frac c{2\nu}\right)^2 \right) w^2\right) = \frac b\nu f'' \left( \hat{v}\right) \hat{v}_x w^2 
$$ 
is strictly positive (resp. negative) for $x\downarrow - \infty$ (resp. $x\uparrow + \infty$). According to the previous part of the proof, $\lim_{|x|\to\infty} w^2 (x) = 0$, hence 
$$
\lim_{|x| \to\infty}  \left( w_x^2 + \left( \frac b\nu f' (\hat{v}) - \left( \frac c{2\nu}\right)^2 \right) w^2\right) \ge  0 
$$ 
so that 
$$ 
w_x^2 + \left( \frac b\nu f' (\hat{v}) - \left( \frac c{2\nu}\right)^2 \right) w^2 > 0   
$$ 
for $x$ at $\pm\infty$. Using  
$w_x = \left( \frac{c}{2 \nu} - \frac{b}{\nu} \frac{f(\hat{v})}{\hat{v}} \right) w$, we conclude that  
$$ 
\left( \frac{c}{2 \nu} - \frac{b}{\nu} \frac{f (\hat{v})}{\hat{v} _x } \right)^2 + \frac{b}{\nu} f^{\prime} (\hat{v}) - \left( \frac{c}{2 \nu} \right)^2 >  0 
$$ 
or equivalently 
\begin{equation} 
\label{eq1:lema1_0} 
\frac{b}{\nu} f^{\prime} (\hat{v}) - \frac{b}{\nu} \frac{f (\hat{v})}{\hat{v} _x}\left( \frac c\nu - \frac{b}{\nu} \frac{ f(\hat{v})}{\hat{v}_x}\right) > 0 \, . 
\end{equation} 
In particular, 
$$ 
\begin{aligned} 
\frac b\nu \frac{d}{dx} \frac{f(\hat{v})}{\hat{v}_x} 
& = \frac b\nu f^{\prime} (\hat{v}) - \frac b\nu \frac{f (\hat{v})}{\hat{v}_x} \frac{\hat{v}_{xx}}{\hat{v}_x} > 0   
\end{aligned} 
$$
so that $\frac{f(\hat{v})}{\hat{v}_x}$ is strictly increasing at $\pm\infty$. 

\smallskip
\noindent  
To prove the remaining identities for $\gamma_\pm$ observe that by l'Hospital's rule we obtain that 
$$ 
\gamma_- = \lim_{x\to -\infty} \frac b\nu \frac{f(\hat{v})}{\hat{v}_x}(x) = \lim_{x\to -\infty} \frac b\nu f^\prime (\hat{v})(x) \frac{\hat{v}_x}{\hat{v}_{xx}} (x) 
= \frac b\nu f^{\prime} (0) \frac 1{\frac c\nu - \gamma_{-}} 
$$ 
or equivalently, $\gamma_{-} \left(\frac c\nu - \gamma_{-}\right) =  \frac b\nu f^\prime (0)$.
Since $\gamma_{-} < 0$ we obtain the assertion. $\gamma_+$ can be computed similarly. 
\end{proof}

\begin{lemma} 
Assume $({\bf A1})-({\bf A2})$. Then Assumption \ref{ass:eigenfct} is satisfied.
\end{lemma}

\begin{proof} 
We have
\begin{align*}
\Vert \Psi\Vert_V^2 = \Vert \Psi\Vert_H^2 + \Vert \Psi_x \Vert_H^2
\leq \left(1 + 2 \frac{c^2}{\nu^2}\right)\int e^{-2\frac c\nu x} \, \hat{v}_x^2 \, dx 
+ 2 \int e^{-2\frac c\nu x}  \, \hat{v}_{xx}^2 \, dx
\end{align*}
The previous Lemma implies that 
$$
\frac{d}{dx} \left( e^{-\left( 2\frac c\nu - \gamma_-\right)x}\hat{v}_x^2 \right)  
= - \left( 2\frac b\nu \frac{f(\hat{v})}{\hat{v}_x} - \gamma_- \right) 
 e^{-\left( 2\frac c\nu - \gamma_-\right)x}\hat{v}_x^2 \ge 0
$$ 
for $x\downarrow-\infty$, hence $M_- := \sup_x e^{-\left( 2\frac c\nu - \gamma_-\right)x} 
\hat{v}_x^2 < \infty$ which implies that 
\begin{equation} 
\label{eq1:lem1_1} 
\int_{-\infty}^x e^{-2\frac c\nu y} \hat{v}_x^2\, dy \le M_- \int_{-\infty}^x e^{-\gamma_- y}\, dy < \infty\qquad \forall\, x\, . 
\end{equation} 
Similarly, 
$$
\frac{d}{dx} \left( e^{-\left( 2\frac c\nu - \gamma_+\right)x}\hat{v}_x^2 \right) 
= - \left( 2\frac b\nu \frac{f(\hat{v})}{\hat{v}_x} - \gamma_+\right) e^{-\left( 2\frac c\nu + 
\gamma_+\right)x}\hat{v}_x^2 \le 0
$$ 
for $x\uparrow\infty$, hence $M_+ := \sup_x e^{-\left( 2\frac c\nu - \gamma_+\right)x}\hat{v}_x^2 < \infty$ which implies that 
\begin{equation} 
\label{eq2:lem1_1} 
\int_x^\infty e^{-2\frac c\nu y} \hat{v}_x^2\, dy \le M_+ \int_x^\infty 
e^{-\gamma_+ y}\, dy < \infty\qquad \forall\, x\, . 
\end{equation} 
Combining \eqref{eq1:lem1_1} and  \eqref{eq2:lem1_1} we obtain that 
$\int e^{-2\frac c\nu x} \, \hat{v}_x^2 \, dx < \infty$. 

\medskip 
\noindent 
Finally, 
\begin{align*}
\int e^{-2\frac c\nu x}  \, \hat{v}_{xx}^2 \, dx \leq \left(\sup_{x\in\R} \frac{\vert \hat{v}_{xx}\vert}{\hat{v}_x}\right)^2  \int e^{-2\frac c\nu x}  \, \hat{v}_{x}^2 \, dx 
\end{align*}
The above supremum can be bounded by
\[
\sup_{x\in\R} \frac{\vert \hat{v}_{xx}\vert}{\hat{v}_x} \leq \frac c\nu + \sup_{x\in\R} \frac{\vert f(\hat{v})\vert}{\hat{v}_x}
\]
which is finite again by the previous Lemma \ref{lem1_0}. 
\end{proof}

We normalise $\Psi$ such that $\lb \Psi, \hat{v}_x \rb =1$.
As mentioned earlier, in the following stability analysis it will turn out to be of advantage to work in weighted measure spaces instead of the unweighted choices of $V$ and $H$. A natural choice for such a measure is 
$$\rho(x):= \frac{\Psi(x)}{\hat{v}_x(x)}= Z \, e^{-\frac c \nu x}$$ where $Z$ denotes the normalising constant. One reason for $\rho$ to be a natural choice is that in the space $L^2(\rho)$ the frozen-wave operator $\lf$ separates $\hat{v}_x$, which due to the identity\newline
\mbox{$\partial_t\hat{v}(x+ct)=c\hat{v}_x(x+ct)$} is the direction of movement of the wave, from its orthogonal complement $\hat{v}_x^\bot$ in the following sense: For $u\in V$
\begin{align*}
0&=\lv u, \lfa\Psi\rvs = \lvs \lf u, \Psi\rv \\
&= - \nu \int u_x \, (\hat{v}_x \rho)_x \, dx
+ b \, \int f'(\hat{v}) u\,  \hat{v}_x \rho \, dx
-c \int  u_x\, \hat{v}_x \rho \, dx\\
&=\lvs \nu \Delta u, \hat{v}_x \rho\rv + b\, \lb f'(\hat{v}) u, \hat{v}_x \rb_{\rho} -c\, \lb u_x, \hat{v}_x \rb_{\rho}\\
&= \lvs \lf u, \hat{v}_x \rho\rv
\end{align*}
Note that for $u\in H^2(\R)$ this is the usual orthogonality in $L^2(\rho)$, i.e. $\lf(H^2(\R))\subset \hat{v}_x^\bot$ in $L^2(\rho)$.

Due to the fact that $\hat{v}_x$ is a zero-eigenvector of $\lf$ perturbations in that direction lead to a random phase shift in the dynamics. To also bound the spread of perturbations in the shape of the wave profile (see Theorem \ref{thm:firstorder}) we need to control the behaviour of the dynamics in directions orthogonal to $\hat{v}_x$, more precisely we assume the following contraction property 
\begin{assumption}\label{ass:spectralgap}
There exist $\kappa>0$, $C_*>0$ such that for $u\in V$ with  $u\rho \in V$
\begin{align*}\label{C1}
\lvs \lf u, u\rho \rv \leq -\kappa \Vert u \Vert_\rho^2+ C_*\langle \hat{v}_x, u \rangle^2_\rho \tag{\bf C1}
\end{align*}
\end{assumption}

i.e. the flow generated by the frozen-wave operator is contracting on the orthogonal complement of $\hat{v}_x$ in $L^2(\rho)$.

Alternatively, this abstract assumption can again be replaced by the additional assumption ({\bf A2}) on the reaction term $f$.
We will prove the contraction property ({\bf C1}) in analogy to \cite[Theorem 1.5]{Stannat14}, where a similar spectral gap inequality for the (unfrozen) operator in the unweighted space $L^2(\R)$ has been shown.

\begin{proposition} Given \textnormal{({\bf A1}) - ({\bf A2})} the frozen-wave operator $\lf$ satisfies Assumption \ref{ass:spectralgap}.
\end{proposition}

\begin{proof}
For $u\in C_c^2(\R)$ we write $u=h\hat{v}_x$. Due to the identity $\lf \hat{v}_x= 0$
we obtain
\begin{align*}
\lf u &= \nu h_{xx} \hat{v}_x +2 \nu h_x \hat{v}_{xx} + \nu h \hat{v}_{xxx}+ b f'(\hat{v}) \hat{v}_x h -c h_x \hat{v}_x - ch\hat{v}_{xx}\\
&= \nu h_{xx} \hat{v}_x +2 \nu h_x \hat{v}_{xx}-c h_x \hat{v}_x 
\end{align*}
and the associated quadratic form
\begin{align*}
\mathcal{E}(h):= - \, \lvs \lf u, u\rho \rv = \nu \int h_x^2\, \hat{v}_x^2 \,Z e^{-\frac c\nu x} \, dx\, .
\end{align*}
Rewriting also \eqref{C1} in terms of $u=h\hat{v}_x$ and setting $w(x)= \hat{v}_x(x)\, e^{-\frac{c}{2\nu} x}$ our aim is to prove an inequality of the type
\begin{align*}
- \nu \int h_x^2 \, w^2 \, dx \leq -\kappa \int h^2 w^2 \, dx + C_\ast\left(\int hw^2 \, dx\right)^2.
\end{align*}

\smallskip 
\noindent 
Recall from the proof of Lemma \ref{lem1_0} that 
$$ 
\frac{d}{dx} \left( e^{-\left( 2\frac c\nu - \gamma_-\right)x}\hat{v}_x^2 \right)  
= - \left( 2\frac b\nu \frac{f(\hat{v})}{\hat{v}_x} - \gamma_- \right) 
 e^{-\left( 2\frac c\nu - \gamma_-\right)x}\hat{v}_x^2 \ge 0
$$ 
for $x\downarrow-\infty$, so that 
$$ 
\int_{-\infty}^x e^{-\frac c\nu y} \hat{v}_x^2\, dy 
\le \int_{-\infty}^x e^{\left( \frac c\nu - \gamma_-\right) y} \hat{v}_x^2\, dy 
\, e^{-\left( 2\frac c\nu - \gamma_-\right)x}\hat{v}_x^2 (x) = \frac 1{\frac c\nu - \gamma_-} 
e^{-\frac c\nu x} \hat{v}_x^2 
$$ 
for $x\downarrow -\infty$. Since $\int_{-\infty}^0 e^{-\frac c\nu y} \hat{v}_x^2\, dy < \infty$ 
and $e^{-\frac c\nu x} \hat{v}_x^2$ is locally bounded from below we can find a finite constant 
$K_-$ such that 
\begin{equation} 
\label{prop:eq1}
\int_{-\infty}^x e^{-\frac c\nu y} \hat{v}_x^2\, dy 
\le K_- e^{-\frac c\nu x} \hat{v}_x^2 \mbox{ for all } x\le 0\, .
\end{equation} 
Similarly, 
$$
\frac{d}{dx} \left( e^{-\left( 2\frac c\nu - \gamma_+\right)x}\hat{v}_x^2 \right)  
= - \left( 2\frac b\nu \frac{f(\hat{v})}{\hat{v}_x} - \gamma_+ \right) 
 e^{-\left( 2\frac c\nu - \gamma_+\right)x}\hat{v}_x^2 \le 0
$$ 
for $x\uparrow +\infty$, so that 
$$ 
\int_x^\infty e^{-\frac c\nu y} \hat{v}_x^2\, dy 
\le \int_x^\infty e^{\left( \frac c\nu - \gamma_+\right) y} \hat{v}_x^2\, dy 
\, e^{-\left( 2\frac c\nu - \gamma_+\right)x}\hat{v}_x^2 (x) = \frac 1{\gamma_+ - \frac c\nu } 
e^{-\frac c\nu x} \hat{v}_x^2 
$$ 
for $x\uparrow +\infty$, thereby using $\gamma_+ - \frac c\nu  = \sqrt{\left( \frac{c}
{2\nu}\right)^2 - \frac b\nu f^\prime (1)} - \frac c{2\nu} > 0$. Since 
$\int_0^\infty e^{-\frac c\nu y} \hat{v}_x^2\, dy < \infty$ 
and $e^{-\frac c\nu x} \hat{v}_x^2$ is locally bounded from below, we can also find a finite constant $K_+$ such that 
\begin{equation} 
\label{prop:eq2}
\int_x^\infty e^{-\frac c\nu y} \hat{v}_x^2\, dy 
\le K_+ e^{-\frac c\nu x} \hat{v}_x^2 \mbox{ for all } x\ge 0\, .
\end{equation}

\smallskip 
\noindent 
Estimate \eqref{prop:eq1} now implies for $h\in C^1_b$ with $h(0) = 0$ that 
$$ 
\begin{aligned} 
\int_{-\infty}^0 h^2 w^2 \, dx & = -2\int_{-\infty}^0 \int_x^0 h_x h\, dy w^2 (x)\, dx 
= -2\int_{-\infty}^0 h_x (y) h (y)\int_{-\infty}^y w^2 (x)\, dx\, dy  \\
& \le 2K_- \int_{-\infty}^0 |h_x (y) h (y)| w^2 (y)\, dy 
 \le \frac 12 \int_{-\infty}^0 h^2 w^2 \, dx + 2K_-^2 \int_{-\infty}^0 h_x^2 w^2 \, dx
\end{aligned} 
$$ 
hence 
$$ 
\int_{-\infty}^0 h^2 w^2 \, dx \le 4K_-^2 \int_{-\infty}^0 h_x^2 w^2 \, dx
$$ 
and similarly, using \eqref{prop:eq2}, 
$$ 
\int_0^\infty h^2 w^2 \, dx \le 4K_+^2 \int_0^\infty h_x^2 w^2 \, dx\, . 
$$ 
Combining both estimates we obtain the weighted Hardy type inequality
\[
\int h^2 w^2 \, dx \leq K_-\vee K_+ \int h_x^2 \, w^2 \, dx
\]
for any $h\in C_b^1(\R)$ with $h(0)=0$. For a general $h\in C_b^1(\R)$ centering allows us to derive the inequality
\[
K_-\vee K_+ \int h_x^2 \, w^2 \, dx \geq \int (h-h(\hat{x}))^2 w^2 \, dx \, \,  .
\]
We introduce the normalising constant $W=\int w^2\, dx$ and the normalised measure $\tilde{w}^2= W^{-1}{w^2}$  to further estimate the above right hand side by
\begin{align*}
 E_{\tilde{w}^2}[(h-h(\hat{x}))^2] \geq \operatorname{Var}_{\tilde{w}^2}(h)
= \int h^2 \tilde{w}^2 \, dx - \left(\int h \tilde{w}^2 \, dx\right)^2\, .
\end{align*}
Altogether, this yields the desired Poincare inequality 
\begin{align*}
\int h^2 w^2 \, dx \leq K_-\vee K_+ \int h_x^2 \, w^2 \, dx + W^{-1} \left(\int hw^2 \, dx\right)^2
\end{align*}
for any $h\in C_b^1(\R)$.
\end{proof}

\begin{lemma}
Under Assumption \ref{ass:spectralgap} $\lf$ generates a contraction semigroup $(P_t^\#)_{t\geq 0}$ on \mbox{$\hat{v}_x^\bot\subset L^2(\rho)$} satisfying
\begin{align}\label{eq:contraction}
\Vert P_t^\# u \Vert_\rho \leq e^{-\kappa t } \Vert u \Vert_ \rho
\end{align}
\end{lemma}

\begin{proof}
Note that $(\lf, C_c^2(\R))$ is symmetric on $L^2(\rho)$ and  the subspace $\hat{v}_x^\bot$ is invariant under $\lf$, i.e. for $u\in \hat{v}_x^\bot$ we have $\langle \lf u, \hat{v}_x \rho \rangle = \langle u, \lfa \Psi \rangle = 0$ and thus $\lf u \in \hat{v}_x^\bot$.
By Assumption \ref{ass:spectralgap}, the operator $\lf$ is bounded from above by
\[
\langle \lf u, u \rangle_\rho \leq - \kappa \Vert u\Vert_\rho^2
\] 
on the orthogonal complement $\hat{v}_x^\bot$. 
It follows that $\lf$ is essentially self-adjoint and its Friedrichs extension $(\lf, \mathcal{D}(\lf))$ generates a symmetric $C_0$-semigroup $(P_t^{\#})_{t\in[0,T]}$ on $\hat{v}_x^\bot$ (see \cite{Kato}).
It is easy to see that Assumption \ref{ass:spectralgap} extends to all $u\in \mathcal{D}(\lf)$. 
In particular,
\begin{align*}
\frac 12 \frac{d}{dt} \Vert P_t^{\#} u\Vert_{\rho}^2 = \lb \lf P_t^{\#}u, P_t^{\#}u \rb_{\rho}
\leq -\kappa  \Vert P_t^{\#}u\Vert_{\rho}^2
\end{align*}
which yields
\[
\Vert P_t^{\#}u\Vert_{\rho}^2 \leq e^{-2\kappa t} \Vert u\Vert_\rho^2
\]
for all $u\in \mathcal{D}(\lf)$ and subsequently for all $u\in \hat{v}_x^\bot$. 
\end{proof}

We would like to highlight that this contraction property is only needed for the asymptotic second moment estimate in Section \ref{sec:prop} but not for the main multiscale decompositions in Theorem \ref{thm:firstorder} and Theorem \ref{thm:limit}.
A similar contraction property can in general not be expected to hold true for the non-autonomous linearisation $(\mathscr{L}_t)$ given in \eqref{def:lin}. But yet we will show that this family of linear operators generates an evolution family on $H^1(1+\rho)$ facilitating a mild solution representation of $u$ on this space.
The family of linear operators $(\mathscr{L}_t),\, t\in[0,T]$, can be seen as operators from \mbox{$H^3(1+\rho) \to H^1(1+\rho)$} satisfying
\begin{align*}
\Vert \lin_t u \Vert_{H^1(1+\rho)} \leq L_* \Vert u \Vert_{H^3(1+\rho)}
\end{align*}

\begin{lemma}
$(\lin_t)$ generates an evolution family $(P_{t,s})_{0\leq s\leq t \leq T}$ on $H^1(1+\rho)$ with 
\begin{align}\label{evol}
\Vert P_{t,s} h \Vert_{H^1(1+\rho)} \leq e^{L_*(t-s)} \Vert h \Vert_{H^1(1+\rho)}
\end{align}
for a constant $L_\ast >0$.
\end{lemma}

\begin{proof}

We first show that the Gaussian semigroup $p_t = e^{t\nu\Delta}$, $t\ge 0$, acts on $H^1(1+\rho)$:  For $u\in H^1(1+\rho)$ 
\begin{align*}
\int (p_tu)^2(x) (1+\rho(x))\, dx \leq \int p_t(u^2)(x) (1+\rho(x)) \, dx = \int u^2 (p_t(1+\rho))(x) \, dx
\end{align*}
using the symmetry of the semigroup. Together with \eqref{ptrho} this enables us to obtain the bounds 
\begin{align*}
\Vert p_tu\Vert_{1+\rho}^2 \leq e^{\frac{c^2}{\nu} t } \Vert u\Vert_{1+\rho}^2
\end{align*}
as well as 
\begin{align*}
\Vert \partial_x p_tu \Vert_{1+\rho}^2&= 
\int (\partial_x p_tu)^2 (1+\rho) \, dx = \int p_t( u_x)^2 (1+\rho)\, dx \leq e^{ \frac{c^2}{\nu}t } \int ( u_x)^2 (1+\rho) \, dx \\
&= e^{ \frac{c^2}{\nu} t } \Vert  u_x \Vert_{1+\rho}^2\, . 
\end{align*}
Set $B(t)u:= b f'(\hat{v}(\cdot +ct))u$, i.e. $\lin_t=\nu \Delta+B(t)$. Defined like this,  $(B(t))_{t\in[0,T]}$ is indeed a family of bounded perturbations of $\nu \Delta$ with
\begin{align*}
\sup_{\Vert u\Vert=1}\Vert B(t)u\Vert_{H^1(1+\rho)}^2 \leq b^2 (\Vert f'(\hat{v})\Vert^2_\infty + \Vert f''(\hat{v})\Vert_\infty^2\Vert \hat{v}_x\Vert_\infty^2)
\end{align*}
uniformly in $t\in[0,T]$. Hence, by \cite[Ch. 5, Theorem 2.3]{Pa} $\lin_t=\nu \Delta+B(t)$ is a stable family of infinitesimal generators. In particular, there exists $L_\ast > 0$ such that the associated evolution family satisfies \eqref{evol}.
\end{proof}

Since $u$ is a variational solution on $H^1(1+\rho)$ it can also be represented as a mild solution
\begin{align*}
u(t)=\eps P_{t,0} \eta + \int_0^t P_{t,s} R(s,u(s)) \, ds
+ \eps \int_0^t P_{t,s} dW(s)
\end{align*}
Conditions for variational solutions to satisfy a mild solution representation are stated in \cite{PreRo}.

\section{Multiscale Analysis}
\subsection{Dynamical equation for phase-adaptation} \label{sec:phase}
To establish a description of the noise-induced phase shift of the wave profile arising due to the nonlinearity of the system the idea is to determine the stochastic phase by dynamically matching the deterministic profile $\hat{v}$ with the stochastic solution $v$. This matching is achieved by minimising the \mbox{$L^2$-distance}
\begin{align}\label{min}
C\mapsto \Vert v(\cdot,t)-\hat{v}(\cdot+ct+C)\Vert_\rho
\end{align}
over all possible phases $C$.
 Again, it turns out to be of advantage to work in the weighted measure space $L^2(\rho)$, where the measure will now be moved with the wave. The following dynamics are designed to point in the direction of the negative gradient of \eqref{min}. Let $m>0$ and consider the (pathwise) ODE
\begin{align}\label{eq:Ct}
\dot{C}^m(t) &= m \, B(t,C^m(t)),\;\; t\in [0,T]\\ 
 C(0) &=0 \notag
\end{align}
where
\begin{align*}
B(t,C) &= \lb   v(t,\cdot)-\hat{v}(\cdot+ct+C), \hat{v}_x(\cdot +ct+C)\rb_{\rho(\cdot+ct+C)}\\
&=\lb   v(t,\cdot)-\hat{v}(\cdot+ct+C), \Psi(\cdot +ct+C)\rb
\end{align*}

Equation \eqref{eq:Ct} can be regarded as an alternative approach to the phase conditions specified by certain algebraic constraints in the classical stability analysis (refer to \cite{Henry}) and has also been introduced in \cite{KrSt},\cite{Lang}, \cite{Thuemmler} and \cite{Stannat14}.

\begin{proposition}[Well-posedness]
$P$-almost surely there exists a unique adapted solution \newline
\mbox{$C\in C^1([0,T])$} of the (pathwise) ODE \eqref{eq:Ct}.
\end{proposition}
\begin{proof} With analogous arguments as in \cite{Stannat14} and \cite{KrSt} $B(t,C)$ is continuous in \mbox{$(t,C)\in[0,T]\times \mathbb{R}$}. It then suffices to show that $B(t,C)$ is Lipschitz continuous in $C$ with a Lipschitz constant independent of $t$: For $C_1, \, C_2 \in \R, \, t>0$

\begin{align*}
&B(t,C_1)-B(t,C_2)= \lb \Psi(\cdot+ct+C_1)-\Psi(\cdot+ct+C_2), v(t)-\hat{v}(\cdot+ct)\rb \\
&+ \lb \Psi(\cdot+ct+C_1), \hat{v}(\cdot+ct)-\hat{v}(\cdot+ct+C_1)\rb -\lb \Psi(\cdot+ct+C_2), \hat{v}(\cdot+ct)-\hat{v}(\cdot+ct+C_2)\rb \\
& = I + II+III , \text{ say.}
\end{align*}
With
\begin{align*}
\vert I\vert \leq \Vert \Psi(\cdot+ct+C_1)-\Psi(\cdot+ct+C_2)\Vert_H \, \Vert u(t)\Vert_H
\leq \Vert \Psi_x\Vert_H \vert C_1-C_2\vert\,  \Vert u\Vert_{C([0,T];H)}
\end{align*}
and 
\begin{align*}
\vert II+III\vert & = \vert \lb \Psi(\cdot+ct), \hat{v}(\cdot+ct-C_1)-\hat{v}(\cdot+ct)\rb -\lb \Psi(\cdot+ct), \hat{v}(\cdot+ct-C_2)-\hat{v}(\cdot+ct)\rb\vert\\
&= \vert \lb \Psi(\cdot+ct), \hat{v}(\cdot+ct-C_1)-\hat{v}(\cdot+ct-C_2)\rb\vert\\
&\leq \Vert \Psi\Vert_H \, \Vert \hat{v}_x\Vert_H \, \vert C_1-C_2\vert
\end{align*}

Therefore, since $u$ is adapted and in $L^0(\O;C([0,T];H))$, there exists a unique adapted process $C\in L^0(\O;C^1([0,T]))$ that solves \eqref{eq:Ct}.
\end{proof}

\noindent Let $\gamma^m(t):= ct + C^m(t)$. The initial representation $v(t)=\hat{v}(\cdot+ct) + u(t)$ can now be replaced by 
\begin{align}
v(t)=\hat{v}(\cdot+\gamma^m(t))+u^m(t),
\end{align}
where $u^m:\O\times [0,T] \to H$ defined by
\[u^m(t) = u(t) + \hat{v}(\cdot+ct) - \hat{v}(\cdot + \gamma^m(t)) = v(t) - \hat{v}(\cdot + \gamma^m(t))\]
is an adapted process in $ L^2(\O;C([0,T];H))\cap L^2(\O\times (0,T);V)$.
 Moreover, $u^m$ is the unique variational solution of the equation
\begin{align}\label{eq:um}
du^m(t) &=  \Big[\nu \Delta u^m(t) +  {G}^m(t,u^m(t)) - (C^m)'(t)\, \hat{v}_x(\cdot +\gamma^m(t)) \Big]\,dt + \varepsilon \, d W(t),
\\ \notag
 & =\Big[\nu \Delta u^m(t)+ b f'(\hat{v}(\cdot+\gamma^m(t)))\, u^m(t) - m\lb \Psi(\cdot +\gamma^m(t)),  u^m(t)\rb \, \hat{v}_x(\cdot +\gamma^m(t))\Big]\,dt\\ \notag
&\quad+ {R}^m(t,u^m(t)) \, dt + \varepsilon \, d W(t)
\end{align}
with initial profile $u^m(0) = u^0$ and 
\begin{align*}
{G}^m (t,u) &=b\left[ f(u+\hat{v}(\cdot +\gamma^m(t)))  - f(\hat{v}(\cdot +\gamma^m(t)))\right]\\  {R}^m(t,u) &= {G}^m (t,u) -b \,f'(\hat{v}(\cdot+\gamma^m(t)))\, u.
\end{align*}
For the Nagumo equation, for instance, the remainder ${R}^m$ is explicitly given by
\[R^m(t,u) =\frac12 f''(\hat{v}(\cdot +\gamma^m(t))) u^2   + \frac16 f'''(\hat{v}(\cdot +\gamma^m(t)))u^3 = \frac12 ((1+a)-3\hat{v}(\cdot +\gamma^m(t)))\, u^2- u^3.\]

\subsection{Multiscale decomposition of the fluctuations}

For the subsequent analysis we demand higher regularity of the reaction function $f$ by assuming that $f\in C^3(\R)$. Let $\rho_t(x)=\rho(x+ct)$ and for $h\in C([0,T], H^1(1+\rho))$ set 
\[
\Vert h\Vert_T = \sup_{t\in [0,T]} \Vert h(t)\Vert_{H^1(1+\rho_t)}
\]
Likewise for $f\in C[0,T]$ define
\[
\vert f\vert_T = \sup_{t\in [0,T]} \vert f(t)\vert.
\]

\noindent Additionally, we again assume ({\bf A1}) - ({\bf A2}) ensuring that $\hat{v}_x\in H^1(1+\rho)$. 
\bigskip

We now derive an SDE for the stochastic perturbation $c^m(t):= \dot{C}^m(t)$ of the wave speed. Note that while the phase adaptation $C^m(t)$ is a process of bounded variation, the resulting adapted wave speed is of unbounded variation.

\begin{lemma}\label{lem:speed}
The dynamically adapted wave speed $c^m(t)= m \,\lb  u^m(t), \Psi(\cdot +\gamma^m(t))\rb$  solves the SDE
\begin{align}\label{eq:speed} 
c^m(t)&=c^m(0)+ m \int_0^t \left( -\, c^m(s)+  c^m(s) \lb u^m(s), \Psi_x(\cdot+\gamma^m(s)\rb +  \lvs {R}^m(s, u^m(s)), \Psi(\cdot + \gamma^m(s))\rv \right) ds\\ \notag
&\quad+ \eps m \int_0^t \lb \Psi(\cdot+\gamma^m(s)), dW(s)\rb\\ \notag
c^m(0)&= \eps m \lb \eta, \Psi\rb
\end{align}
\end{lemma}

\begin{proof}
Applying It\^o's lemma we obtain
\begin{align*}
c^m(t)-c^m(0)&= m \int_0^t \lvs \nu \Delta u^m(s)+ b f'(\hat{v}(\cdot+\gamma^m(s)))\, u^m(s), \Psi(\cdot+\gamma^m(s))\rv\, ds \\
&\quad - m \int_0^t \lb \, \lb \Psi(\cdot +\gamma^m(s)),  u^m(s)\rb \, \hat{v}_x(\cdot +\gamma^m(s)), \Psi(\cdot+\gamma^m(t))\rb\, ds \\
&\quad+ m \int_0^t \lvs {R}^m(s,u^m(s)), \Psi(\cdot+\gamma^m(s))\rv\, ds 
+ \eps m \int_0^t \lb dW(s) , \Psi(\cdot+\gamma^m(s))\, \rb  \\
&\quad + m \int_0^t\lb u^m(s), \Psi_x(\cdot+\gamma^m(s))(c+c^m(s))\rb\, ds\\
&=  m \int_0^t \lvs \nu \Delta u^m(s)+ b f'(\hat{v}(\cdot+\gamma^m(s)))\, u^m(s) - c\partial_x u^m(s), \Psi(\cdot+\gamma^m(s))\rv\, ds \\
&\quad -m\int_0^t c^m(s)\, ds+ m \int_0^t\lvs {R}^m(s,u^m(s)), \Psi(\cdot+\gamma^m(s))\rv\, ds  \\
&\quad + m \int_0^t c^m(s)\, \lb u^m(s), \Psi_x(\cdot+\gamma^m(s)\rb \, ds + \eps m \int_0^t\lb dW(s) , \Psi(\cdot+\gamma^m(s))\, \rb\\
&= m \int_0^t \lv u^m(s, \cdot -\gamma^m(s)), \lfa\Psi\rvs\, ds \\
&\quad -m\int_0^t c^m(s)\, ds + m \int_0^t c^m(s)\, \lb u^m(s), \Psi_x(\cdot+\gamma^m(s)\rb \, ds \\
&\quad + m \int_0^t \lvs {R}^m(s,u^m(s)), \Psi(\cdot+\gamma^m(s))\rv\, ds + \eps m \int_0^t\lb dW(s) , \Psi(\cdot+\gamma^m(s))\, \rb
\end{align*}
\end{proof}

In order to investigate the dynamics on different scales of the noise strength we first formally identify the highest order terms in \eqref{eq:speed} as well as \eqref{eq:um}. Expecting that both $C^m$ and $u^m$ are of order $\eps$ leads us to define $c^m_0$ to be the unique strong solution of
\begin{align}\label{eq:c0}
dc_0^m(t)&= -m \, c_0^m(t) \, dt + m \lb \Psi(\cdot+ct), dW(t) \rb, \quad t\in[0,T]\\\notag
c_0^m(0)&= m \lb \eta, \Psi \rb
\end{align}

and $u_0^m \in L^2(\O;C([0,T];H))\cap L^2(\O\times (0,T);V)$ to be the unique variational solution to
\begin{align}\label{eq:u0}
du_0^m(t)&= \left[\lin_t u_0^m(t) - c_0^m(t) \hat{v}_x(\cdot+ ct)\right] \, dt + dW(t) , \quad t\in[0,T]\\\notag
u_0^m(0)&= \eta
\end{align}
with $u_0^m \in L^2(\Omega, C([0,T], H^1(1+\rho)))$. This regularity holds true since $u_0^m$ has the mild solution representation
\[
u_0^m(t) = P_{t,0} \eta -\int_0^t P_{t,s} c_0^m(s) \hat{v}_x(\cdot+cs) \, ds + \int_0^t P_{t,s} dW(s)
\]
and $\hat{v}_x\in \ho$.
We define the first order phase adaptation by $C_0^m(t)= \int_0^t c_0^m(s) \, ds$ and the first order phase by $\gamma_0^m(t)=ct + \eps C_0^m(t)$. 
For $\eps>0 $ and $q\in [0,1]$ set
\begin{align}\label{tauqeps}
\tau_{q,\eps} = \inf\{t\in[0,T] : \, \Vert u(t) \Vert_{H^1(1+\rho_t)} \geq \eps^{1-q}\} 
\end{align}
where $u$ is the solution from Proposition \ref{prop:ex} and
\[
\tau_{q,\eps}^m=\inf\{t\in[0,T] : \vert C_0^m(t)\vert \geq \eps^{-q}\}
\]

\begin{theorem}\label{thm:firstorder}
Let $q<\frac 1 2 $. On $\{\tau_{q,\eps} \wedge \tau_{q,\eps}^m = T\}$ the stochastic travelling wave $v$ can be decomposed into
\[
v(t)=\hat{v}(\cdot+ct+\eps C_0^m(t))+ \eps u_0^m(t) + \eps r^m(t)
\]
with 
\[
\Vert r^m\Vert_T \leq \alpha(T) \eps^{1-2q}
\]
where the constant $\alpha(T)$ is independent of $m$ and $\eps$.
Moreover,
\[
\lim_{\eps \to 0} P[\tau_{q,\eps} \wedge \tau_{q,\eps}^m = T] =1
\]
\end{theorem}

\begin{proof}
Let $\tilde{u}_0^m(t):= v(t)-\hat{v}(\cdot+\gamma_0^m(t))= u(t)+\hat{v}(\cdot+ct)-\hat{v}(\cdot+\gamma_0^m(t))$. By Taylor's formula there exists $\xi(t,x)$ with $\vert \xi(t,x)\vert \leq \eps \vert C_0^m(t)\vert$ uniformly in $x$ such that
\[
\tilde{u}_0^m(t)=u(t)+\hat{v}_x(\cdot +ct+\xi(t,\cdot))\,\eps C_0^m(t).
\]
Hence, on $\{\tau_{q,\eps} \wedge \tau_{q,\eps}^m = T\}$ using Remark \ref{lem:rho}
\begin{align*}
\Vert \tilde{u}_0^m(t)\Vert_{H^1(1+\rho_t)}&\leq \Vert u(t)\Vert_{H^1(1+\rho_t)}+ \eps \vert C_0^m(t)\vert \Vert \hat{v}_x(\cdot+ct+\xi(t))\Vert_{H^1(1+\rho_t)}\\
&\leq \eps^{1-q}+ \eps^{1-q}  \Vert \hat{v}_x\Vert_{H^1(1+\rho(\cdot- \xi(t)))}\\
&\leq \eps^{1-q}+ \eps^{1-q} (1\vee e^{M\eps^{1-q}})^\frac 12 \Vert \hat{v}_x\Vert_{H^1(1+\rho)}\leq C_1 \, \eps^{1-q}
\end{align*}

\noindent for a constant $C_1>1$. The remainder process
\[ 
r^m(t)=\frac 1 \eps \Big(v(t)-\hat{v}(\cdot +ct+\eps C_0^m(t))\Big)-u_0^m(t)
\]
is a variational solution of the following pathwise evolution equation:

\begin{align*}
dr^m(t)&= \lin_tr^m(t)\, dt \\
&\quad+ \frac b\eps\Big(f(\hat{v}(\cdot+\gamma_0^m(t))+\tilde{u}_0^m(t))-  f(\hat{v}(\cdot+\gamma_0^m(t)))-  f'(\hat{v}(\cdot+\gamma_0^m(t)))\, \tilde{u}_0^m(t)\Big)dt\\
&\quad + \frac b \eps \Big(f'(\hat{v}(\cdot+\gamma_0^m(t)))-f'(\hat{v}(\cdot+ct))\Big)\, \tilde{u}_0^m(t) \, dt\\
&\quad + c_0^m(t)\Big(\hat{v}_x(\cdot+ct)-\hat{v}_x(\cdot+\gamma_0^m(t))\Big) \, dt\\
&=: \big[\lin_tr^m(t)+r_1^m(t)+r_2^m(t)+r_3^m(t)\big]\, dt
\end{align*}
with $r^m(0)=0$. Thus, it can be represented as a mild solution
\[
r^m(t)=\int_0^t P_{t,s}\left(r_1^m(s)+r_2^m(s)+r_3^m(s)\right)\, ds \;\;\in H^1(1+\rho)\;
\]
Using \eqref{evol} and Remark \ref{lem:rho} (i) we estimate
\begin{align*}
&\Vert r^m(t)\Vert_{H^1(1+\rho_t)} \\
&\leq \int_0^t e^{L_*(t-s)}\big(\Vert r_1^m(s)\Vert_{H^1(1+\rho_s)}+\Vert r_2^m(s)\Vert_{H^1(1+\rho_s)}\big)\, ds
+\Big\Vert \int_0^t P_{t,s}r_3^m(s)\, ds\Big\Vert_{H^1(1+\rho_t)}
\end{align*}

\noindent The first part is bounded by applying condition $({\bf B3})$ as follows:

\begin{align*}
\frac \eps b \, \vert r_1^m(t,x)\vert &\leq \eta_2\,(1+\vert \tilde{u}_0^m(t,x)\vert) \, \vert \tilde{u}_0^m(t,x)\vert^2\\
&\leq \eta_2 \left(\Vert\tilde{u}_0^m(t)\Vert_{\infty}+
\Vert\tilde{u}_0^m(t)\Vert_{\infty}^2\right)\vert \tilde{u}_0^m(t,x)\vert\\
&\leq \eta_2 \left(\Vert\tilde{u}_0^m(t)\Vert_{H^1(1+\rho_t)}+
\Vert\tilde{u}_0^m(t)\Vert_{H^1(1+\rho_t)}^2\right)\vert \tilde{u}_0^m(t,x)\vert
\end{align*}

\noindent Therefore, on $\{\tau_{q,\eps} \wedge \tau_{q,\eps}^m = T\}$

\begin{align*}
\Vert r_1^m(t)\Vert_{1+\rho_t}\leq  b \, \eta_2 \,(C_1 \eps^{1-q}+C_1^2 \eps^{2-2q})\, C_1 \eps^{-q}\leq 2 b\, \eta_2 \,C_1^3\, \eps^{1-2q}
\end{align*}
Furthermore, using Taylor's formula there exists an intermediate point $\xi(t,x)$ with \newline
\mbox{$\vert \xi(t,x)\vert \leq \vert \tilde{u}_0^m(t,x)\vert$} such that
\begin{align*}
\partial_x r_1^m(t,x)&=\frac b {2\eps} \, f'''\big(\hat{v}(x+\gamma_0^m(t))+\xi(t,x)\big)\,  \tilde{u}_0^m(t,x)^2\, \hat{v}_x(x+\gamma_0^m(t))\\
&\quad +\frac b \eps \Big(f'( \hat{v}(x+\gamma_0^m(t))+\tilde{u}_0^m(t,x))-f'(\hat{v}(x+\gamma_0^m(t)))\Big) \partial_x\tilde{u}_0^m(t,x)
\end{align*}

Note that even though $f''$ and $f'''$ are not assumed to be globally bounded, we can control the above expression using local bounds on $\{\tau_{q,\eps} \wedge \tau_{q,\eps}^m = T\}$. Since $\hat{v}\in[0,1]$ and 
$$\vert \xi(t,x)\vert \leq \vert \tilde{u}_0^m(t,x)\vert 
\leq \Vert \tilde{u}_0^m(t)\Vert_{H^1(1+\rho_t)}\leq C_1 \eps^{1-q}<1 $$
for $\eps$ small enough, we know that $\hat{v}(x)+\tilde{u}_0^m(t,y) \in[-1,2] $ for all $t>0,\, x,y\in\R$, and therefore
\begin{align*}
\Vert \partial_x r_1^m(t)\Vert_{1+\rho_t}&\leq \frac b{2\eps} \Vert f'''\Vert_{\infty,[-1,2]}\Vert \hat{v}_x\Vert_{\infty}\Vert \tilde{u}_0^m(t)\Vert_{H^1(1+\rho_t)}\Vert \tilde{u}_0^m(t)\Vert_{1+\rho_t}\\
&\quad + \frac b\eps \Vert f''\Vert_{\infty,[-1,2]} \Vert \tilde{u}_0^m(t)\Vert_{H^1(1+\rho_t)}\Vert \partial_x \tilde{u}_0^m(t)\Vert_{1+\rho_t}\\
&\leq \frac b{2} \Vert f'''\Vert_{\infty,[-1,2]}\Vert \hat{v}_x\Vert_{\infty}\, C_1^2 \, \eps^{1-2q}+  b \Vert f''\Vert_{\infty,[-1,2]} \, C_1^2 \, \eps^{1-2q}\\
&= C_1^2\, b\left(\frac 1 2 \Vert f'''\Vert_{\infty,[-1,2]}\Vert \hat{v}_x\Vert_{\infty}+\Vert f''\Vert_{\infty,[-1,2]}\right) \, \eps^{1-2q}
\end{align*}

\noindent The second part can be controlled by
\begin{align*}
\left\vert r^m_2(t,x)\right\vert &\leq \frac b \eps\Vert f''\Vert_{\infty,[0,1]} \vert \hat{v}(x+\gamma_0^m(t))-\hat{v}(x+ct)\vert \, \vert \tilde{u}_0^m(t,x)\vert\\
&\leq b \, \Vert f''\Vert_{\infty,[0,1]} \Vert \hat{v}_x\Vert_{\infty} \vert C_0^m(t)\vert\, \vert \tilde{u}_0^m(t,x)\vert
\end{align*}
and
\begin{align*}
\vert \partial_x r_2^m(t,x)\vert &\leq \frac b\eps \Big\vert f''(\hat{v}(x+\gamma_0^m(t)))\hat{v}_x(x+\gamma_0^m(t)))-f''(\hat{v}(x+\gamma_0^m(t)))\hat{v}_x(x+ct)\Big \vert\, \vert \tilde{u}_0^m(t,x)\vert\\
&\quad +\frac b\eps \Big\vert f''(\hat{v}(x+\gamma_0^m(t)))\hat{v}_x(x+ct))-f''(\hat{v}(x+ct))\hat{v}_x(x+ct)\Big \vert \,\vert \tilde{u}_0^m(t,x)\vert\\
&\quad + \frac b \eps \Big\vert f'(\hat{v}(x+\gamma_0^m(t)))-f'(\hat{v}(x+ct))\Big \vert \, \vert \partial_x \tilde{u}_0^m(t,x)\vert\\
&\leq \frac b\eps \Vert f''\Vert_{\infty,[0,1]} \vert \hat{v}_x(x+\gamma_0^m(t))-\hat{v}_x(x+ct)\vert \, \vert \tilde{u}_0^m(t,x)\vert\\
&\quad + \frac b \eps \Vert \hat{v}_x\Vert_{\infty} \vert f''(\hat{v}(x+\gamma_0^m(t)))-f''(\hat{v}(x+ct))\vert \, \vert \tilde{u}_0^m(t,x)\vert\\
&\quad + \frac b\eps \Vert f''\Vert_{\infty, [0,1]} \vert \hat{v}(x+\gamma_0^m(t))-\hat{v}(x+ct)\vert \, \vert \partial_x\tilde{u}_0^m(t,x)\vert\\
&\leq  b  \left( \Vert f''\Vert_{\infty,[0,1]} \Vert \hat{v}_{xx}\Vert_{\infty}+\Vert f''\Vert_{\infty,[0,1]} \Vert \hat{v}_x\Vert_{\infty}^2  \right) \vert C_0^m(t)\vert\,\vert \tilde{u}_0^m(t,x)\vert\\
&\quad +  b \Vert f''\Vert_{\infty,[0,1]} \Vert \hat{v}_x\Vert_{\infty}  \vert C_0^m(t)\vert \, \vert \partial_x\tilde{u}_0^m(t,x)\vert
\end{align*}

\noindent such that on $\{\tau_{q,\eps} \wedge \tau_{q,\eps}^m = T\}$

\begin{align*}
\Vert r_2^m(t)\Vert_{H^1(1+\rho_t)}^2&=\Vert r_2^m(t)\Vert_{1+\rho_t}^2+\Vert \partial_x r_2^m(t)\Vert_{1+\rho_t}^2\\
&\leq b^2 \left( \Vert f''\Vert_{\infty,[0,1]}^2 \Vert \hat{v}_x\Vert_{\infty}^2 + 2 \Vert f''\Vert_{\infty,[0,1]}^2 \Vert \hat{v}_{xx}\Vert_{\infty}^2+ 2 \Vert f''\Vert_{\infty,[0,1]}^2 \Vert \hat{v}_x\Vert_{\infty}^4\right) \eps^{-2q}\, \Vert \tilde{u}_0^m(t)\Vert^2_{1+\rho_t} \\
&\quad +  b^2 \Vert f''\Vert_{\infty,[0,1]}^2 \Vert \hat{v}_x\Vert_{\infty}^2  \eps^{-2q} \, \Vert \partial_x\tilde{u}_0^m(t)\Vert^2_{1+\rho_t}\\
&\leq  b^2 \Vert f''\Vert_{\infty,[0,1]}^2 \left(  \Vert \hat{v}_x\Vert_{\infty}^2 + 2 \Vert \hat{v}_{xx}\Vert_{\infty}^2+ 2 \Vert \hat{v}_x\Vert_{\infty}^4\right)C_1^2 \, \eps^{2-4q}\\
&=:C_2^2 \, \eps^{2-4q}
\end{align*}
For the last part set
\[
R_3^m(t):= - \frac 1\eps \left(\hat{v}(\cdot+\gamma_0^m(t))-\hat{v}(\cdot +ct)-\eps C_0^m(t)\hat{v}_x(\cdot+ct)\right)
\]
satisfying
\[
\left(\frac d{dt}-c\partial_x\right)R_3^m(t)=r_3^m(t).
\]
Thus, we obtain
\begin{align*}
\int_0^t P_{t,s}r_3^m(s) \, ds = R_3^m(t)+\int_0^t P_{t,s}(\lin_s-c\partial_x)R_3^m(s)\, ds
\end{align*}
and
\begin{align*}
\bigg\Vert \int_0^t P_{t,s}r_3^m(s) \, ds\bigg\Vert_{H^1(1+\rho_t)}
&\leq \Vert R_3^m(t)\Vert_{H^1(1+\rho_t)}+ \int_0^t \Vert P_{t,s}(\lin_s-c\partial_x)R_3^m(s)\Vert_{H^1(1+\rho_s)} \, ds\\
&\leq \Vert R_3^m(t)\Vert_{H^1(1+\rho_t)}\\
&\quad + \int_0^t e^{L_*(t-s)}\Vert\lin_s-c\partial_x\Vert_{\lin(H^3(1+\rho_s), H^1(1+\rho_s))}\, \Vert R_3^m(s)\Vert_{H^3(1+\rho_s)} \, ds\\
&= I+II
\end{align*}
By Taylor's theorem there exists $\xi(t,x)$ with $\vert \xi\vert \leq \eps \vert C_0^m\vert$ such that the order of the first summand can be estimated by
\begin{align*}
\Vert R_3^m(t)\Vert_{H^1(1+\rho_t)}&\leq \frac\eps 2\,  \vert C_0^m(t)\vert^2 \Vert \hat{v}_{xx}(\cdot+ct+\xi(t))\Vert_{H^1(1+\rho_t)}\\
&\leq \frac 12 (1\vee e^{M\eps^{1-q}})^{\frac 12} \Vert \hat{v}_{xx}\Vert_{H^1(1+\rho)}\, \eps^{1-2q} =:C_3 \, \eps^{1-2q}
\end{align*}
For estimating the order of the second summand one needs to control also higher derivatives of $R_3^m$:
\begin{align*}
\Vert R_3^m(t)\Vert_{H^3(1+\rho_t)}^2 = \Vert R_3^m(t)\Vert_{H^1(1+\rho_t)}^2 + \Vert \partial_{xx}R_3^m(t)\Vert_{L^2(1+\rho_t)}^2 +\Vert \partial_{xxx} R_3^m(t)\Vert_{L^2(1+\rho_t)}^2
\end{align*}
with
\begin{align*}
&\eps \partial_{xx} R_3^m(t)= \hat{v}_{xx}(\cdot +ct+\eps C_0^m(t))-\hat{v}_{xx}(\cdot+ct)-\eps C_0^m(t)\hat{v}_{xxx}(\cdot+ct)\\
&\eps \partial_{xxx} R_3^m(t)= \hat{v}_{xxx}(\cdot +ct+\eps C_0^m(t))-\hat{v}_{xxx}(\cdot+ct)-\eps C_0^m(t)\hat{v}_{xxxx}(\cdot+ct)
\end{align*}
Note that differentiating \eqref{tw} yields
\[
\nu\, \hat{v}_{xxxxx}= c\hat{v}_{xxxx} -f'(\hat{v}) \, \hat{v}_{xxx} - 3 f''(\hat{v}) \, \hat{v}_{xx}\, \hat{v}_{x} -f'''(\hat{v})\hat{v}_x^3
\]
implying that $\hat{v}\in C^5$ if $f\in C^3$. Again applying Taylor's theorem there exist $\xi_1(t), \xi_2(t)$ with $\vert \xi_{1,2}(t)\vert \leq \eps \vert C_0^m(t)\vert$ such that
\begin{align*}
\Vert \partial_{xx}R_3^m(t)\Vert_{L^2(1+\rho_t)}&\leq \eps\,  \frac{\vert C_0^m(t)\vert^2}{2} (1\vee e^{M\eps^{1-q}})^{\frac 12}\Vert \hat{v}_{xxxx}\Vert_{1+\rho}\\
&\leq (1\vee e^{M\eps^{1-q}})^{\frac 12} \frac{\Vert\hat{v}_{xxxx}\Vert_{1+\rho}}{2}\,  \eps^{1-2q}\leq C_4 \, \eps^{1-2q}
\end{align*}

\noindent for a constant $C_4>0$ and

\begin{align*}
\Vert \partial_{xxx}R_3^m(t)\Vert_{L^2(1+\rho_t)}\leq \eps \, \frac{\vert C_0^m(t)\vert^2}{2} (1\vee e^{M\eps^{1-q}})^{\frac 12} \Vert \hat{v}_{xxxxx}\Vert_{L^2(1+\rho)} \leq C_5 \, \eps^{1-2q}
\end{align*}
with $C_5>0$. Altogether we obtain
\begin{align*}
II \leq C \sup_{t\in[0,T]} \Vert \lin_t-c\partial_x\Vert_{\lin(H^3(1+\rho_t), H^1(1+\rho_t))}\frac{e^{L_*t}-1}{L_*} \, \eps^{1-2q}
\end{align*}
for a constant $C>0$ independent of $\eps$. It remains to show that in the small-noise limit the above order estimate holds for $P$-almost all paths $\omega\in \Omega$. If $\tau_{q,\eps} < T$ then, due to continuity, for $t_0=:\tau_{q,\eps}(\omega)$ we obtain
\begin{align*}
\eps^{1-q}&=\Vert u(t_0)\Vert_{H^1(1+\rho_{t_0})}\\
&\leq \eps \Vert C_0^m(t_0) \hat{v}_x(\cdot +ct_0+\xi(t_0))+u_0^m(t_0)\Vert_{H^1(1+\rho_{t_0})} + \eps \Vert r^m(t_0)\Vert _{H^1(1+\rho_{t_0})}\\
&\leq \eps \Vert C_0^m(t_0) \hat{v}_x(\cdot +ct_0+\xi(t_0))+u_0^m(t_0)\Vert_{H^1(1+\rho_{t_0})} + \alpha(T)\eps^{2-2q}
\end{align*}
and therefore, using Markov's inequality
\begin{align*}
P[\tau_{q,\eps} < T]&\leq P[\, \Vert C_0^m(t_0) \hat{v}_x(\cdot +ct_0+\xi(t_0))+u_0^m(t_0)\Vert_{H^1(1+\rho_{t_0})} \geq \eps^{-q}(1-\alpha(T)\eps^{1-q})]\\
&\leq \frac {\eps^{2q}}{2(1-\alpha(T)\eps^{1-q})^2} 
\left( E\left[\vert C_0^m\vert_T^2\right] \,(1\vee e^{M \eps^{1-q}}) \Vert \hat{v}_x\Vert_{H^1(1+\rho)}^2  + E\left[\Vert u_0^m\Vert_T^2\right]\right)\\
&\longrightarrow 0\quad  \text{ as } \eps \to 0.
\end{align*}
Likewise, the second stopping time converges as follows
\begin{align*}
P[\tau_{q,\eps}^m < T]\leq P[\, \vert C_0^m\vert _T \geq \eps^{-q}]\leq \eps^{2q} E\left[\vert C_0^m\vert _T^2\right] \underset{\eps \to 0}{\longrightarrow} 0 .
\end{align*}
\end{proof}

\subsection{Immediate relaxation}

From the definition of the stochastic phase adaptation process $C^m$ it is clear that the initial goal of minimising the $L^2(\rho)$-distance between $v$ and $\hat{v}$ for every time $t\in[0,T]$ can only approximately be achieved when choosing a finite relaxation rate $m$. Although Theorem \ref{thm:firstorder} shows that already for finite $m$ a multiscale decomposition into processes of the expected order can be installed, we are interested in investigating the case of so-called immediate relaxation, i.e. the limit $m\to\infty$, and show that in that case indeed a rigorous minimisation (on relevant orders of the noise strength) is achieved. 
As an alternative to this description of a stochastic phase one could try to adapt the analysis developed in \cite{InLa} for a generalised framework of neural field equations, where the dynamics of the local minimum of \eqref{min} are explicitly described by an SDE up to a certain stopping time. This approach does not incorporate a gradient-descent procedure. As pointed out by the authors, studying instead the behaviour of the global minimum of \eqref{min} would be much more complicated, since its dynamics will be highly discontinuous. To our knowledge, this has not been investigated for bistable reaction-diffusion equations.
Below we will adapt the methods from \cite{Lang}. It will turn out that the limit phase adaptation is a process of unbounded variation behaving almost like a Brownian motion, which is in accordance with the phenomenological description of stochastic travelling waves in nonlinear systems developed in \cite{BrWe} for stochastic neural field equations.
\medskip

\noindent Let $\Pi_t$ denote the projection onto the orthogonal complement of $\hat{v}_x(\cdot+ct)$ in $L^2(\rho_t)$, i.e.
\[
\Pi_t h= h-\lb h, \hat{v}_x(\cdot+ct)\rb_{\rho(\cdot+ct)} \,\hat{v}_x(\cdot+ct)
\]

\begin{lemma} \label{lem:conv}
Define the processes $C_0$ and $u_0$ as
\begin{align*}
C_0(t)&=\lb \eta, \Psi\rb +  \int_0^t \lb \Psi(\cdot+cs)\, , \,dW(s)\rb \quad \text{ for } t>0\\
C_0(0)&=0
\end{align*}
and
\begin{align*}
u_0(t)&=P_{t,0}\Pi_0\eta + \int_0^t P_{t,s}\Pi_s dW(s)\quad \text{ for } t>0\\
u_0(0)&=\eta
\end{align*}
Then for any $\delta>0$ almost surely

\[
\sup_{\delta\leq t \leq T} \vert C_0^m(t)-C_0(t)\vert \underset{m\to \infty}{\longrightarrow} 0
\]
as well as
\[
\sup_{\delta\leq t \leq T}  \Vert u_0^m(t)-u_0(t)\Vert_{H^1(1+\rho_t)}\underset{m\to \infty}{\longrightarrow} 0
\]
\end{lemma}

\begin{proof}
Integrating \eqref{eq:c0} yields
\[
c_0^m(t)=e^{-mt} m \lb \eta, \Psi\rb + \int_0^t e^{-m(t-s)} m \lb \Psi(\cdot+cs), dW(s)\rb
\]
and therefore
\begin{align}\label{c0m}
C_0^m(t)=(1-e^{-mt})\lb \eta, \Psi\rb + \int_0^t \left(1-e^{-m(t-s)}\right)\lb \Psi(\cdot+cs), dW(s)\rb
\end{align}
With this, the difference between the approximate relaxation and the immediate relaxation process is given by
\begin{align*}
C_0(t)-C_0^m(t)= e^{-mt}\lb \eta, \Psi\rb + \int_0^t e^{-m(t-s)}\lb \Psi(\cdot+cs), dW(s)\rb =: e^{-mt}\lb \eta, \Psi\rb + S_t
\end{align*}
For the martingale term an integration by parts leads to
\begin{align*}
S_t&= \lb \Psi(\cdot+ct), W(t)\rb - \int_0^t m e^{-m(t-s)}\lb \Psi(\cdot+cs), W(s)\rb +c\, e^{-m(t-s)}\lb \Psi_x(\cdot+cs), W(s)\rb \, ds
\end{align*}
Now, the function
$t \mapsto \lb \Psi(\cdot+ct), W(t)\rb$
is H\"older continuous for any $\beta<\frac 12$ almost surely, i.e.
\[
M_\beta(T,\omega):= \sup_{\vert t-s \vert \leq T} \frac{\big\vert \lb \Psi(\cdot+ct), W(t)\rb - \lb \Psi(\cdot+cs), W(s)\rb\big \vert}{\vert t-s\vert^\beta} < \infty \quad\text{  a.s.}
\]
Thus,
\begin{align*}
S_t&\leq c \,\bigg\vert \int_0^t e^{-m(t-s)} \lb \Psi_x(\cdot+cs), W(s)\rb \, ds \bigg\vert
+M_\beta(T,\omega) \int_0^t me^{-m(t-s)}(t-s)^\beta \, ds\\
&\quad +e^{-mt} \vert \lb \Psi(\cdot+ct, W(t)\rb \vert \\
&\leq \frac c m \, \Vert \Psi_x\Vert \sup_{0\leq t\leq T} \Vert W(t)\Vert+\frac{M_\beta(T,\omega)}{m^\beta}\, \Gamma(1+\beta)+ e^{-mt}\, \Vert \Psi\Vert \sup_{0\leq t\leq T}\Vert W(t)\Vert
\end{align*}
where $\Gamma$ denotes the Gamma function
\[
\Gamma(t)=\int_0^\infty x^{t-1}e^{-x} \, dx.
\]

\noindent Hence, 
\begin{align*}
&\sup_{\delta \leq t \leq T} \vert C_0(t)-C_0^m(t)\vert\\
&\qquad \leq e^{-m\delta} \Vert \Psi \Vert \left (\Vert \eta \Vert +\sup_{0\leq t \leq T} \Vert W(t)\Vert\right) + \frac{M_\beta(T,\omega)}{m^\beta}\, \Gamma(1+\beta)
+ \frac c m \, \Vert \Psi_x\Vert \sup_{0\leq t\leq T} \Vert W(t)\Vert\\
&\qquad \underset{m\to \infty}{\longrightarrow } 0 \quad \text{ a.s.}
\end{align*}

Note that since $\lin_t \hat{v}_x(\cdot+ct)=0$ we have
$P_{t,s} \hat{v}_x(\cdot+cs)=\hat{v}_x(\cdot+ct)$. Thus, the mild solution representation of the first-order fluctuations $u_0^m$ is given by
\begin{align*}
u_0^m(t)&=P_{t,0}\eta -\int_0^t c_0^m(s) P_{t,s}\hat{v}_x(\cdot+cs)\, ds + \int_0^t P_{t,s} \,dW(s)\\
&=P_{t,0} \Pi_0\eta + \lb \eta,\hat{v}_x\rb_{\rho}\, P_{t,0}\hat{v}_x - \int_0^t c_0^m(s) \hat{v}_x(\cdot +ct) \, ds + \int_0^t P_{t,s} \, dW(s)\\
&= P_{t,0} \Pi_0\eta + P_{t,0}\lb \eta, \Psi\rb\, \hat{v}_x -\hat{v}_x(\cdot+ct)\,  C_0^m(t)+\int_0^t P_{t,s}\Pi_s dW(s)\\
&\quad+\int_0^t P_{t,s} \hat{v}_x(\cdot+cs) \lb \Psi(\cdot+cs), dW(s)\rb\\
&= u_0(t)+\hat{v}_x(\cdot+ct)\left(\lb\eta,\Psi\rb + \int_0^t \lb \Psi(\cdot+cs), dW(s)\rb -C_0^m(t)\right)\\
&= u_0(t)+\hat{v}_x(\cdot+ct)\left(C_0(t)-C_0^m(t)\right)
\end{align*}
Using this we obtain
\begin{align*}
\sup_{\delta \leq t \leq T} \Vert u_0^m(t)-u_0(t)\Vert_{H^1(1+\rho_t)} \leq \sup_{\delta \leq t \leq T} \vert C_0^m(t)-C_0(t)\vert \, \Vert \hat{v}_x\Vert_{H^1(1+\rho)} \underset{m\to \infty}{\longrightarrow} 0 \quad \text{ a.s.}
\end{align*}
\end{proof}

Indeed, in contrast to \cite{InLa} where the existence of a phase-adaptation process has been shown up to a stopping time, Lemma \ref{lem:conv} provides us with effective formulae for the first-order stochastic phase-adaptation and fluctuations for all times $t\in[0,T]$. 
We now show that passing over to the limit from finite to immediate phase relaxation preserves the previous multiscale decomposition (Theorem \ref{thm:firstorder}).

\begin{theorem}\label{thm:limit}
Let $\tau_{q,\eps}$ be defined as in \eqref{tauqeps} and set
\[
\tau_{q,\eps}^\infty=\inf\{t\in[0,T] : \, \vert C_0(t)\vert \geq \eps^{-q}\} \wedge T
\]
Then on $\{\tau_{q,\eps}\wedge \tau_{q,\eps}^\infty=T\}$ the following multiscale decomposition of the stochastic travelling wave $v$ holds:
\[
v(t) = \hat{v}(\cdot +ct+\eps C_0(t))+\eps u_0(t) + \eps r(t)
\]
with
\[
\Vert r\Vert_T\leq \alpha(T)\, \eps^{1-2q}
\]
where $\alpha(T)$ is a positive constant.
Moreover, in the small-noise limit the above representation holds for almost every path $\omega\in \Omega$, i.e.
\[
P[\tau_{q,\eps}\wedge \tau_{q,\eps}^\infty=T]\underset{\eps \to 0}{\longrightarrow} 1.
\]
\end{theorem}

\begin{proof}
Let $t<\tau_{q,\eps}\wedge \tau_{q,\eps}^\infty$. Performing an integration by parts in \eqref{c0m} we obtain
\begin{align*}
C_0^m(t)&= (1-e^{-mt})\lb \eta, \Psi\rb
+ \int_0^t m e^{-m(t-s)}\lb \Psi(\cdot+cs, W(s)\rb \, ds\\
&\quad -c\int_0^t (1-e^{-m(t-s)})\lb \Psi_x(\cdot+cs), W(s)\rb \, ds
\end{align*}
which yields for $0<\delta<t$
\[
\vert C_0^m(t)\vert_\delta \leq \vert \lb \eta, \Psi\rb\vert + \Vert \Psi\Vert \, \Vert W\Vert_\delta + c\delta \Vert \Psi_x\Vert \, \Vert W\Vert_\delta \underset{\delta\to 0}{\longrightarrow} \vert \lb \eta, \Psi\rb\vert
\]

Since $\vert C_0(t)\vert \approx \vert  \lb \eta, \Psi\rb\vert$ for $t$ close to $0$, we know that $\vert C_0^m(t)\vert_\delta < \eps^{-q}$. Furthermore, for any $\delta>0$ 
\[
\sup_{\delta\leq s \leq t} \vert C_0^m(s)\vert \underset{m\to\infty}{\longrightarrow} \sup_{\delta \leq s \leq t} \vert C_0(s)\vert < \eps^{-q}.
\] 
This implies $\{\tau_{q,\eps}\wedge \tau_{q,\eps}^\infty=T\}\subseteq \{\tau_{q,\eps}\wedge \tau_{q,\eps}^m=T\}$ for $m$ sufficiently large. Hence, applying Theorem \ref{thm:firstorder} and Lemma \ref{lem:conv} we obtain
\begin{align*}
&\Vert \eps r(t)\Vert_{H^1(1+\rho_t)} = \Vert v(t)-\hat{v}(\cdot + ct+\eps C_0(t))-\eps u_0(t)\Vert_{H^1(1+\rho_t)}\\
&\leq \Vert v(t)-\hat{v}(\cdot + ct+\eps C_0^m(t))-\eps u_0^m(t)\Vert_{H^1(1+\rho_t)}\\
&\quad+ \Vert \hat{v}(\cdot + ct+\eps C_0^m(t))-\hat{v}(\cdot + ct+\eps C_0(t))\Vert_{H^1(1+\rho_t)}
+\eps \Vert u_0^m(t)-u_0(t)\Vert_{H^1(1+\rho_t)}\\
&\leq \alpha(T)\, \eps^{2-2q} + \eps \vert C_0^m(t)-C_0(t)\vert \, \Vert \hat{v}_x\Vert_{H^1(1+\rho)}(1\vee e^{M\eps^{1-q}})^{\frac 12}
+\eps \Vert u_0^m(t)-u_0(t)\Vert_{H^1(1+\rho_t)}\\
&\underset{m\to\infty}{\longrightarrow}\alpha(T) \eps^{2-2q} \quad \text{a.s.}
\end{align*}
In the limit $\eps\to 0$ the above order estimate holds for almost every path $\omega\in \Omega$, i.e.
\begin{align*}
P[\tau_{q,\eps}\wedge \tau_{q,\eps}^\infty=T] \geq
1-P[\tau_{q,\eps}<T]-P[\tau_{q,\eps}^\infty<T] \underset{\eps\to 0}{\longrightarrow} 1
\end{align*} 
with analogous arguments as in Theorem \ref{thm:firstorder}.
\end{proof}

\subsection{Statistical and geometrical properties of first-order phase adaptation and of fluctuations}\label{sec:prop}
We would like to compare our stochastic phase adaptation $C_0$ to the phase description obtained in \cite{BrWe}, where the phenomenon of stochastic wave propagation has been (formally) investigated for (nonlocal) stochastic neural field equations. They concluded that to first order of the noise strength the stochastic perturbation of the phase is a Brownian motion.
Taking a look at the variance of the immediate phase adaptation we see that for $t>0$
\[
\operatorname{Var}(C_0(t)) =\operatorname{Var} \left(\int_0^t \lb \Psi(\cdot+cs), dW(s)\rb \right) = \int_0^t \lb \Psi(\cdot+cs), Q\Psi(\cdot+cs)\rb \, ds \approx \lb \Psi, Q\Psi\rb \, t
\]
showing that $C_0$ is roughly diffusive if $Q$ is ``almost'' translation invariant. Note that strict translation invariance is excluded since $Q$ is of finite trace. Thus, the above statement can only be an heuristic approximative description.
Our first-order fluctuations are indeed orthogonal to the direction of movement of the wave: For $t>0$
\begin{align}\label{eq:ortho}\notag
&\lb u_0(t), \hat{v}_x(\cdot+ct)\rb_{\rho_t} = 
\lb P_{t,0}\Pi_0\eta, \Psi(\cdot+ct)\rb+ \lb \int_0^t P_{t,s}\Pi_s dW(s), \Psi(\cdot+ct)\rb\\
& =
\lb \Pi_0\eta, P_{t,0}^*\Psi(\cdot+ct)\rb + \int_0^t \lb P_{t,s}^* \Psi(\cdot+ct), \Pi_s dW(s)\rb \\ \notag
&= \lb \Pi_0\eta, \Psi \rb + \int_0^t \lb \Psi(\cdot+cs), \Pi_s dW(s) \rb = 0\, . 
\end{align}  
Likewise, in the frozen wave setting $u_0^\#$ is orthogonal to $\hat{v}_x$ in $L^2(\rho)$. As stated in Subsection \ref{sec:frozen} the frozen wave operator $\lf$ generates a contraction semigroup on $\hat{v}_x^\bot$, which allows for the mild solution representation
\begin{align*}
u_0^\#(t)= P_t^\#\Pi_0\eta + \int_0^t P_{t-s}^\#\Phi_s\Pi_s dW(s)
= P_t^\#\Pi_0\eta + \int_0^t P_{t-s}^\#\Pi_0\Phi_s dW(s)\, . 
\end{align*}
Using the contraction property \eqref{eq:contraction} yields
\begin{align*}
\Vert u_0(t)\Vert_{\rho_t} = \Vert u_0^\#(t)\Vert_\rho \leq 
e^{-\kappa t}\Vert \eta\Vert_\rho + \Big\Vert \int_0^t P_{t-s}^\# \Pi_0\Phi_s dW(s)\Big\Vert_\rho\, . 
\end{align*}
This allows us to bound the expectation by
\begin{align*}
E\left[\Vert u_0(t)\Vert_{\rho_t}^2\right]\leq 2 e^{-2\kappa t} \Vert \eta\Vert_\rho + 2 \int_0^t \Vert P_{t-s}^\# \Pi_0 \Phi_s \sqrt{Q} \Vert_{L_2(L^2(1+\rho), L^2(\rho))}^2 \, ds.
\end{align*}

\noindent For a given orthonormal basis $(e_k)_{k\geq 1}$ of $L^2(1+\rho)$ we expand
\begin{align*}
&\Vert P_{t-s}^\# \Pi_0 \Phi_s \sqrt{Q} \Vert_{L_2(L^2(1+\rho), L^2(\rho))}^2 = \sum_k \Vert P_{t-s}^\# \Pi_0 \Phi_s \sqrt{Q}e_k \Vert_{L^2(\rho)}^2 \leq e^{-2 \kappa (t-s)} \sum_k \Vert \sqrt{Q} e_k \Vert_{\rho_s}^2\\ 
& \leq e^{-2 \kappa (t-s)} \Vert \sqrt{Q}\Vert_{L_2(L^2(1+\rho), L^2(\rho))}^2.
\end{align*}
Thus, we obtain the (asymptotic) second moment estimate
\begin{align*}
E\left[\Vert u_0^\#(t)\Vert_{\rho}^2\right]\leq 2 e^{-2 \kappa t} \Vert \eta \Vert_\rho + 2 \frac{\Vert \sqrt{Q}\Vert_{L_2(L^2(1+\rho), L^2(\rho))}^2}{2 \kappa}(1-e^{-2 \kappa t})\underset{t\to \infty}{\longrightarrow}  \frac{\Vert \sqrt{Q}\Vert_{L_2(L^2(1+\rho), L^2(\rho))}^2}{\kappa}.
\end{align*}

\subsection{Minimisation}
It is still open to verify that the above choice of $C_0$ indeed realises the declared objective of minimising the distance between the stochastic wave $v$ and all possible translations of the deterministic profile $\hat{v}$, thus offering an apt description for a stochastic phase. Since all relevant dynamics have been considered on a scale of order $\eps$, it is natural to also investigate the minimisation property on this scale.
\begin{proposition}
For $t<\tau_{q,\eps}\wedge \tau_{q,\eps}^\infty$ the function $a\mapsto \Vert v-\hat{v}(\cdot+ct+\eps a)\Vert_{\rho_t}$
is locally minimal to order $\eps$ at $a=C_0(t)$. 
\end{proposition}

\begin{proof}
Applying Theorem \ref{thm:limit} we obtain
\begin{align*}
&\frac 12 \, \frac{d}{da}\Big\vert_{a=C_0(t)} \Vert v(t)-\hat{v}(\cdot+ct+\eps a)\Vert_{\rho_t}^2 = - \eps^2 \lb u_0(t)+r(t), \hat{v}_x(\cdot+\gamma_0(t))\rb_{\rho_t}\\
&=-\eps^2\big( \lb \hat{v}_x(\cdot+ct), u_0(t)\rb_{\rho_t} + \lb \hat{v}_x(\cdot+ct+\eps C_0(t))-\hat{v}_x(\cdot+ct), u_0(t)\rb_{\rho_t}+
\lb \hat{v}_x(\cdot +ct), r(t)\rb_{\rho_t}\\
&\quad + \lb \hat{v}_x(\cdot+ct+\eps C_0(t))-\hat{v}_x(\cdot+ct), r(t)\rb_{\rho_t}\big)
= o(\eps^2)\, . 
\end{align*}
Here we used that $\lb \hat{v}_x(\cdot+ct), u_0(t)\rb_{\rho_t}=0$ as shown in \eqref{eq:ortho}. For the second derivative we obtain
\begin{align*}
\frac 12 \, \frac{d^2}{da^2}\Big\vert_{a=C_0(t)} \Vert v(t)-\hat{v}(\cdot+ct+\eps a)\Vert_{\rho_t}^2 = \eps^2 \Vert \hat{v}_x(\cdot+ct)\Vert_{\rho_t}^2 + o(\eps^2) >0\, . 
\end{align*}
\end{proof}

\medskip 
\textbf{Acknowledgement:} We would like to thank Mark Veraar for fruitful discussions during an early stage of this work. 

\nocite{ErmTer}
\nocite{Lang}

\bibliographystyle{ams-pln}

\bibliography{literature}

\end{document}